\documentclass[final]{siamltex}
\usepackage{graphicx}%
\usepackage{multirow}%
\usepackage{amsmath,amssymb,amsfonts}%
\usepackage{mathrsfs}%
\usepackage[title]{appendix}%
\usepackage{xcolor}%
\usepackage{textcomp}%
\usepackage{manyfoot}%
\usepackage{booktabs}%
\usepackage{listings}%
\usepackage{bm}
\usepackage{enumerate}
\usepackage{float} 
\usepackage{subfigure} 
\usepackage{caption} 
\usepackage[ruled,linesnumbered]{algorithm2e}
\SetKwInput{Pre}{Precondition}
\SetKwInput{Post}{Postcondition}
\newcommand{\difFrac}[2]{\frac{\dif #1}{\dif #2}}
\newcommand{\BCJ}{\mathrm{BCJ}}
\newcommand{\dif}{\mathrm{d}}
\newcommand{\ii}{\mathbf{i}}

\newcommand{\sgn}{\mathrm{sign}}
\newcommand{\pdfFrac}[2]{\frac{\partial #1}{\partial #2}}

\title
{
  Lagrangian Particle Classification and
  Lagrangian Flux Identities
  for a Moving Hypersurface
}

\author{Lingyun Ding\thanks{Department of Mathematics,
University of California Los Angeles,
405 Hilgard Avenue,
Los Angeles, California, 90095 United States
(\tt dingly@g.ucla.edu).}\and 
Shuang Hu\thanks{Zhejiang University,
866 Yuhangtang Road, 
Hangzhou, Zhejiang Province, 310058 China
({\tt 12335013@zju.edu.cn}).}\and 
Baiyun Huang\thanks{
  School of Mathematical Sciences,
  Zhejiang University
  ({\tt huangbaiyun1@126.com}).}\and 
Qinghai Zhang\thanks{Corresponding author. 
School of Mathematical Sciences,
Zhejiang University,
({\tt qinghai@zju.edu.cn}). 
}
}

\begin{document}
\maketitle
\begin{abstract}
  For a moving hypersurface in the flow
 of a nonautonomous ordinary differential equation
 in $n$-dimensional Euclidean spaces, 
 the fluxing index of a passively-advected Lagrangian particle
 is the total number of times it crosses the moving hypersurface
 within a time interval.
The problem of Lagrangian particle classification
 is to decompose the phase space into flux sets, 
 equivalence classes of Lagrangian particles at the initial time.
In the context of scalar conservation laws, 
 the problem of Lagrangian flux calculation (LFC)
 is to find flux identities that relate
 the Eulerian flux of a scalar through the moving hypersurface, 
 a spatiotemporal integral over the moving surface
 in a given time interval, 
 to spatial integrals over donating regions
 at the initial time of the interval.
In this work,
 we implicitly characterize flux sets via topological degrees,
 explicitly construct donating regions,
 prove the equivalence of flux sets and donating regions,
 and establish two flux identities; 
 these analytical results constitute
 our solutions to the aforementioned problems. 
Based on a flux identity suitable for numerical calculation, 
 we further proposed a new LFC algorithm, 
 prove its convergence,
 and demonstrate its efficiency,
 good conditioning, and high-order accuracy
 by results of various numerical tests.
  


\end{abstract}

\section{Introduction}
\label{sec:introduction}

For a time-dependent velocity field $\mathbf{u}(\mathbf{x},t)$ 
 that is continuous in time and Lipschitz continuous in space, 
 the nonautonomous ordinary differential equation (ODE)
\begin{equation}
  \label{eq:velODE}
  \frac{\dif\, \mathbf{x}}{\dif\, t} = \mathbf{u}(\mathbf{x},t)
\end{equation}
 admits a unique solution for any initial time $t_0$
 and any initial position $\mathbf{p}(t_0)\in\mathbb{R}^m$.
This uniqueness gives rise to a flow map
 $\phi: \mathbb{R}^{m}\times \mathbb{R}\times
 \mathbb{R}\rightarrow \mathbb{R}^{m}$
 that maps the initial position $\mathbf{p}(t_0)$
 of a Lagrangian particle $\mathbf{p}$,
 the initial time $t_0$,
 and the time increment $k$
 to $\mathbf{p}(t_0+k)$, 
 the position of $\mathbf{p}$ at time $t_0+k$, 
\begin{equation}
  \label{eq:PhiExpression}
  \phi_{t_{0}}^{k}(\mathbf{p}) := \mathbf{p}(t_0+k)
  = \mathbf{p}(t_0) + \int_{t_0}^{t_{0}+k}
  \mathbf{u}(\mathbf{p}(t),t)\dif t.
\end{equation} 

For fixed $t_{0}\in\mathbb{R}$ and $k>0$, 
 the homeomorphism
 $\phi_{t_{0}}^{+k}: \mathbb{R}^m \rightarrow \mathbb{R}^m$
 satisfies 
 $\phi^{-k}_{t_{0}+k}(\phi_{t_0}^{+k}(\mathbf{p}))=\mathbf{p}$ 
 and $\phi_{t_0}^{+k}(\phi^{-k}_{t_{0}+k}(\mathbf{x}))=\mathbf{x}$, 
 i.e., $\phi_{t_{0}+k}^{-k}$ is the inverse of $\phi_{t_{0}}^{+k}$.


A common characteristic curve of the flow map is the \emph{pathline},
 a curve generated by following a particle $\mathbf{p}$
 within a time interval $[t_0, t_0+ k]$: 
\begin{equation}
  \label{eq:pathline}
  \Phi_{t_{0}}^{+k}(\mathbf{p})
  =\left\{\phi_{t_{0}}^{+\tau k}(\mathbf{p})\;|\;\tau\in(0,1)\right\}.
\end{equation}

Let ${\mathcal S}(t)$ be a homotopy class of oriented hypersurfaces,
 each having co-dimension one in $\mathbb{R}^m$.
If $\mathbf{u}$, $\mathbf{p}(t_0)$, $[t_0,t_0+k]$,
 and ${\mathcal S}(t)$ are given a priori,
 one can follow the particle $\mathbf{p}$
 to count the number of crossings
 of the pathline $\Phi_{t_{0}}^{+k}(\mathbf{p})$
 to ${\mathcal S}(t)$,
 with the sign of each crossing given by the inner product
 of the relative velocity
 $\mathbf{u}(\mathbf{p}(t_{\times}), t_{\times})
 - \partial_t{\mathcal S}(t_{\times})$ 
 and the normal vector of ${\mathcal S}(t_{\times})$
 at the crossing point $\mathbf{p}(t_{\times})$. 
This number is called 
 the \emph{fluxing index of the Lagrangian particle} of $\mathbf{p}$
 with respect to $\mathbf{u}$ and ${\mathcal S}(t)$
 within $[t_0,t_0+k]$; 
 see Definition \ref{def:fluxingIndex} and Figure \ref{fig:fluxingIndex}
 for more details. 
 
Conversely,
 given a velocity $\mathbf{u}$,
 a moving hypersurface ${\mathcal S}(t)$,
 and a time interval $[t_0,t_0+k]$,
 we want to decompose $\mathbb{R}^m$
 into \emph{flux sets},
 equivalence classes with the equivalence relation
 as the fluxing index of Lagrangian particles
 marked at the initial time $t_0$.
This problem is called \emph{Lagrangian particle classification}. 

On top of the above problem, consider a scalar field
 $f: \mathbb{R}^m\times\mathbb{R}\rightarrow\mathbb{R}$
 that satisfies the conservation law
 with respect to $\mathbf{u}$ in (\ref{eq:velODE}), 
\begin{equation}
  \label{eq:ConservationLaw}
  \partial_{t} f (\mathbf{x},t)
  + \nabla \cdot ( \mathbf{u} (\mathbf{x},t) f (\mathbf{x},t))=0.
\end{equation}

The \emph{Eulerian flux} (or \emph{flux}) of $f$ through
 a moving hypersurface $\mathcal{S} (t)$
 within a time interval $[t_0,t_{0}+k]$ can be expressed as
\begin{equation}
  \label{eq:EulerianFluxIntegral}
  \int_{t_0}^{t_0+k} \int_{\mathcal{S}(t)}
  f(\mathbf{x},t)
  \bigl[\mathbf{u} (\mathbf{x},t)- \partial_{t}\mathcal{S}(t)\bigr]
  \cdot \mathbf{n} (\mathbf{x},t)\, \dif \mathbf{x} \dif t,
\end{equation}
where
$\mathbf{n} (\mathbf{x},t)$ is the unit outward normal vector
of $\mathcal{S}(t)$ at $\mathbf{x}$;
see Definition \ref{def:OutwardNvecForSurf}. 

The notion of fluxes is ubiquitous
in transport, mixing, and other physical processes
such as Lagrangian coherent structures;
see, e.g., 
\cite{shadden2005definition,lipinski2010ridge,zhang2014oceanic,lin2017high,ding2021enhanced}.
The calculation of flux integrals is also of much importance
in developing numerical schemes such as
the finite volume (FV) methods
\cite{leveque2002finite,sarmin1963stability,schiesser2009compendium}
and volume-of-fluid (VOF) methods
\cite{zhang2013family}. 
Furthermore, error estimates of flux calculations
 are fundamental in the numerical analysis of
 FV methods \cite{bell1988unsplit,colella1990multidimensional}
 and VOF methods \cite{zhang2013donating,zhang2013highly}.

To calculate the Eulerian flux integral (\ref{eq:EulerianFluxIntegral}),
 one can solve the conservation law \eqref{eq:ConservationLaw}
 by an FV method
 to obtain the evolution of the scalar $f$
 over the time interval $[t_0, t_0+k]$
 and then evaluate (\ref{eq:EulerianFluxIntegral})
 via numerical quadrature.
This process probably also involves
 computing intersections of the moving hypersurface
 to the fixed control volumes
 and the interpolation of $f$ to desired surface patches. 
As such, the whole process can be very time-consuming.

To avoid
 numerically solving the conservation law (\ref{eq:ConservationLaw}), 
 one way is to convert
 the Eulerian flux integral (\ref{eq:EulerianFluxIntegral})
 to a Lagrangian flux integral via the flux identity
 \begin{equation}
   \label{eq:FluxIdentity}
   \int_{t_0}^{t_0+k} \int_{\mathcal{S} (t)}
   f(\mathbf{x},t)
   \bigl[\mathbf{u} (\mathbf{x},t) -\partial_{t}\mathcal{S} (t)\bigr]
   \cdot \mathbf{n} (\mathbf{x},t)
   \dif \mathbf{x} \dif t
   = \sum\limits_{n\in \mathbb{Z}\setminus\{0\}}
   n \int_{\mathcal{D}^{n}_{\mathcal{S}} (t_0,k)}^{}
   f (\mathbf{p},t_{0})\mathrm{d} \mathbf{p},
 \end{equation}
 where $\mathcal{D}^{n}_{\mathcal{S}} (t_0,k)$
 is the \emph{donating region (DR) of ${\mathcal S}(t)$
 of index $n$} \cite{zhang2013donating},
 a subset of $\mathbb{R}^m$ at the initial time $t_0$
 such that (\ref{eq:FluxIdentity}) holds;
 see Definition \ref{def:DR} for a precise definition
 and \cite[Fig. 4.3]{zhang2019lagrangian}
 for several illustrations. 
The right-hand side (RHS) is called the \emph{Lagrangian flux} 
 of (\ref{eq:EulerianFluxIntegral})
 because each integral domain $\mathcal{D}^{n}_{\mathcal{S}} (t_0,k)$
 can be constructed
 via tracing characteristic curves of the flow map $\phi$.

In (\ref{eq:FluxIdentity}), 
 the spatiotemporal integral on the left-hand side
 is converted to a spatial integral at the initial time on the RHS,
 obviating the time dependence of the scalar $f$
 in calculating its flux. 
Thus the flux identity (\ref{eq:FluxIdentity}) is useful in analyzing
 local truncation errors of unsplit multidimensional FV algorithms, 
 and it applies even when $f$ is discontinuous in space; 
 see \cite{zhang2013family} for such an analysis.
  
The problem of \emph{Lagrangian flux calculation} (LFC)
 consists of two parts:
 (i) selecting and proving a flux identity
 similar to (\ref{eq:FluxIdentity})
 and (ii) designing an efficient and accurate algorithm
 to calculate the flux (\ref{eq:EulerianFluxIntegral}).
On the one hand,
 we desire to select the flux identity
 whose form is best suited for numerical calculation;
 on the other hand, 
 the design of LFC algorithms should fully exploit
 the theoretical insights provided by the flux identity.
If our primary interest is not
 the evolution of the scalar over the entire computational domain
 but the dynamics in a \emph{local} region, 
 LFC could be much more flexible and efficient
 than the aforementioned FV approach for computing the Eulerian flux.

In his seminal work,
 Zhang \cite{zhang2013donating}
 gave an explicit construction of DRs in two dimensions
 and showed that (\ref{eq:FluxIdentity})
 holds if the time increment $k>0$
 is sufficiently small. 
Based on this analysis,
 he also proposed an LFC algorithm \cite{zhang2013highly}
 for solving scalar conservation laws
 with semi-Lagrangian methods. 
Later,
 he \cite{zhang2015generalized}
 removed the restrictive assumption of $k$ being sufficiently small.
Utilizing the concept of winding numbers, 
 these works are restricted to two dimensions; 
 the common steps are to 
\begin{enumerate}[(LFC2D.1)]
\item construct a closed curve called \emph{the generating curve of DRs}
 from a velocity field $\mathbf{u}(\mathbf{x},t)$,
 a time interval $[t_0, t_0+k]$, 
 and a static curve $\widetilde{LN}$,
\item define DRs as the equivalence classes
 of locations of Lagrangian particles at $t_0$
 with the equivalence relation being
 the winding numbers with respect to the generating curve in (LFC2D.1),
\item prove the index-by-index equivalence of flux sets
  to DRs in (LFC2D.2). 
\end{enumerate}

More recently,
 Karrasch and colleagues
 \cite{Karrasch2016Lagrangian,hofherr2018lagrangian}
 defined DRs from an alternative viewpoint,
 gave a proof of the flux identity (\ref{eq:FluxIdentity})
 for moving surfaces in two and higher dimensions, 
 and proposed an LFC algorithm \cite[Algorithm 1]{hofherr2018lagrangian}.
However,
 LFC with the flux identity (\ref{eq:FluxIdentity})
 necessitates computing intersections
 of the boundaries of DRs with different indices, 
 which can be arbitrarily ill-conditioned.
In addition, their LFC algorithm 
 only works in two dimensions with second-order accuracy
 and the generalizations to higher dimensions
 and higher accuracy are not obvious.

To overcome this ill-conditioning,
 Zhang and Ding \cite{zhang2019lagrangian}
 proposed a two-dimensional LFC algorithm,
 hereafter referred to as LFC-2019, 
 based on another flux identity
 \begin{equation}
   \label{eq:FluxIdentity1}
   \sum_{n\in \mathbb{Z}\setminus\{0\}}
   n \int_{\mathcal{D}^{n}_{\widetilde{LN}}(t_0,k)}
   f (x,y,t_{0})\,\dif x\dif y
   = \oint_{\gamma_{\mathcal{D}} (t_0,k)}
   {F} ({x},y,t_0)\dif y,
 \end{equation}
 where $\gamma_{\mathcal{D}}(t_{0},k)$ is
 the generating curve of the DR in (LFC2D.1)
 for the fixed curve $\widetilde{LN}$,
 the function
 $F(x,y,t_{0}):=\int_{\xi}^{x} f(s,y,t_{0})\,\dif s$
 satisfies
 $\frac{\partial F}{\partial x} = f $ 
 for any fixed real number $\xi$.
Although $\gamma_{\mathcal D}$ 
 can be considered as the boundary of the DR
 in (\ref{eq:FluxIdentity}),
 one can not deduce (\ref{eq:FluxIdentity1})
 directly from (\ref{eq:FluxIdentity}) and Green's theorem:
 $\gamma_{\mathcal D}$ may be self-intersecting
 while the boundary of the integral domain
 in Green's theorem must be simple closed.
Thanks to the RHS of (\ref{eq:FluxIdentity1})
 being a {line integral},
 LFC-2019 only consists of constructing
 the generating curve $\gamma_{\mathcal D}$
 and integrating $F$ along $\gamma_{\mathcal D}$.
In particular,
 it is free of computing intersections of DR boundaries. 
Consequently,
 LFC-2019 reduces to numerically solving ODEs
 and calculating weighted sums of function values of $f$
 at the initial time $t_0$.
Via approximating the generating curves with splines, 
 Zhang and Ding showed that
 LFC-2019 is well conditioned
 and can be second-, fourth-, and sixth-order accurate. 
See \cite{zhang2013donating,zhang2019lagrangian}
 for more details on the background and applications of LFC.

In this work,
 we solve the general problem of LFC for a moving hypersurface
 in Euclidean spaces $\mathbb{R}^m$ where $m\ge 2$.
More specifically, 
\begin{enumerate}[(A)]
\item we characterize flux sets
  as equivalence classes of topological degrees
  of a certain function, which is composed from 
  the flow map and the parametrization of the moving hypersurface; 
\item we generalize the flux identity 
  (\ref{eq:FluxIdentity1}) to three and higher dimensions
  by customizing the divergence theorem
  and the Reynolds transport theorem
  to self-intersecting hypersurfaces; 
\item we propose, based on (B), 
  a simple, highly accurate, and well conditioned 
  LFC algorithm for moving surfaces in three dimensions.
\end{enumerate}

To the best of our knowledge,
 the LFC algorithm in (C)
 is the first of its kind that applies to three dimensions
 and generalizes in a straightforward way
 to higher dimensions.
This generalization from two dimensions to higher dimensions, however, 
 is not straightforward, due to several main difficulties.
First, the winding number on which (LFC2D.1--3) rely
 is a concept dedicated to the complex plane. 
Although it is known that topological degrees
 are generalizations of winding numbers,
 it is nontrivial to formulate and prove flux identities
 with this abstract notion. 

Second,
 due to Cauchy's theorem in complex analysis, 
 the differentiability of a map $\mathbb{R}^2\rightarrow \mathbb{R}^2$
 immediately implies its analyticity or conformality, 
 which dictates that the orientation of a closed curve be preserved
 under the action of a diffeomorphism in two dimensions.
Therefore,
 a Jordan curve can be oriented \emph{extrinsically}
 according to its bounded and unbounded complements of the plane.
This extrinsic orientation does not hold in higher dimensions
 because a diffeomorphism $\mathbb{R}^m\rightarrow \mathbb{R}^m$
 with $m>2$ needs not preserve extrinsic orientations
 of a closed hypersurface.
For example, Smale \cite{smale1959classification}
 showed the existence of such a three-dimensional diffeomorphism
 that turns a sphere inside out.
Consequently,
 in order to generalize LFC to three and higher dimensions,
 one needs to be very careful
 in defining the orientation of the moving hypersurface. 


The rest of this paper is organized as follows.
In Section \ref{sec:preliminaries}, 
 we introduce notations,
 collect relevant definitions and results,
 and prepare the reader for subsequent sections.
In particular, we give a coherent exposition 
 on how to generalize the two-dimensional winding number to
 the higher-dimensional concept of topological degrees.
The importance of correctly orienting hypersurfaces 
 and cycles is emphasized by the existence of sphere eversion
 in Section \ref{sec:sphere-eversion}. 
In Section \ref{sec:analysis},
 we orient hypersurfaces in an \emph{intrinsic} manner,
 implicitly characterize flux sets via topological degrees,
 explicitly construct flux sets via generating cycles, 
 customize the divergence theorem and the Reynolds transport theorem
 for cycles with potential self-intersections,
 and prove the flux identity
 that is best suited for numerical flux calculations.
In Section \ref{sec:algorithm},
 we exploit the flux identity
 to propose a new LFC algorithm in three dimensions,
 elaborating on its algorithmic details. 
In Section \ref{sec:tests},
 various numerical tests are performed
 to validate the flux identity and to verify the new LFC algorithm.
Results of these tests demonstrate the efficiency,
 the good conditioning, 
 and the second-, fourth-, and sixth-order accuracy
 of the proposed LFC algorithm.
In Section \ref{sec:conclusion},
 we conclude this paper with several research prospects.



\section{Preliminaries}
\label{sec:preliminaries}

\subsection{Winding numbers}
\label{sec:winding-numbers}

A \emph{closed curve} 
 is the image of a continuous function
 $\gamma: [0, 2\pi]\rightarrow \mathbb{R}^2$
 with $\gamma(0)=\gamma(2\pi)$.
A \emph{Jordan curve} is a closed curve $\Gamma$
 whose parametrized function $\gamma$
 is injective on $[0,2\pi)$.
The Jordan curve theorem states that
 $\mathbb{R}^2\setminus \Gamma$ consists of
 only two components, one bounded and one unbounded,
 with $\Gamma$ being their common boundary.
A Jordan curve $\Gamma$ is \emph{positively oriented}
 if $\BCJ(\Gamma)$, the bounded complement of $\Gamma$,
 always lies to the left of an observer who traverses $\Gamma$
 according to $\gamma$;
 otherwise it is \emph{negatively oriented}. 

A closed curve can be viewed as the image
 of a Jordan curve $\Gamma$ under a continuous map
 $\chi: \mathbb{R}^2\rightarrow\mathbb{R}^2$.
The \emph{winding number of an oriented closed curve}
 $\chi(\Gamma)\subset \mathbb{R}^2\simeq \mathbb{C}$
 around a point $a\in \mathbb{R}^2\setminus \chi(\Gamma)$
 is the number of times it encircles $a$,
 i.e., 
 \begin{equation}
   \label{eq:windingNumber}
   w(\chi(\Gamma),a) := \frac{1}{2\pi\ii}
   \int_{\chi(\Gamma)} \frac{\dif z}{z-a}, 
 \end{equation}
 where $\ii=\sqrt{-1}$
 and $w(\chi(\Gamma),z)$ 
 is a constant integer in each connected component of
 $\mathbb{R}^2\setminus \chi(\Gamma)$
 \cite[p. 203]{rudin86:_real_compl_analy}.
In particular, we have 
  \begin{equation}
    \label{eq:windingNumberJordan}
    w(\Gamma,a) =
    \begin{cases}
      +1 & \text{if $a\in \BCJ(\Gamma)$
        and $\Gamma$ is positively oriented},
      \\
      -1 & \text{if $a\in \BCJ(\Gamma)$ and $\Gamma$ is negatively oriented},
      \\
      0  & \text{if $a$ belongs to the unbounded complement of $\Gamma$}.
    \end{cases}
  \end{equation}

A \emph{free homotopy}  in $\mathbb{R}^2$ between two closed curves
 parametrized as $\gamma_1$ and $\gamma_2$ 
 is a function $H_\chi:[0,1]^2\rightarrow \mathbb{R}^2$
 such that $H_\chi(\theta,0)=\gamma_1(2\pi\theta)$
 and $H_\chi(\theta,1)=\gamma_2(2\pi\theta)$
 for all $\theta$,
 and $H_\chi(0,t)=H_\chi(1,t)$ for all $t$.
Then $\gamma_1$ and $\gamma_2$ 
 are said to be \emph{freely homotopic}
 in $\mathbb{R}^2$. 
The most essential characterization of winding numbers is
 
\begin{theorem}
  \label{thm:Hopf}
  Let a point $a\in \mathbb{R}^2$ be given.
  Two closed curves $\gamma_1$ and $\gamma_2$
  are freely homotopic in $\mathbb{R}^2\setminus\{a\}$
  if and only if
  $w(\gamma_1, a) = w(\gamma_2, a)$.
\end{theorem}

Denote by $\overline{{\mathcal D}}$
 the closure of a point set ${\mathcal D}$. 
The Hopf theorem and (\ref{eq:windingNumber}) lead to 
 
\begin{theorem}[Argument principle]
  \label{thm:argumentPrincipleAnalytic}
  For a positively oriented Jordan curve $\Gamma$,
  an analytic map 
  $\chi:\overline{\textrm{BCJ}(\Gamma)}
  \rightarrow\mathbb{C}$, 
  and a point $a\not\in \chi(\Gamma)$, 
  we have 
  \begin{equation}
    \label{eq:argumentPrincipleAnalytic}
    \begin{array}{l}
      w(\chi(\Gamma),a) = \sum_{z_j\in \chi^{-1}(a)} m_j,
    \end{array}
  \end{equation}
  where $m_j$ is the algebraic multiplicity of the preimage $z_j$. 
\end{theorem}

Recall that a point $x_0\in \Omega\subset \mathbb{R}^m$
 is a \emph{critical point}
 of a ${\mathcal C}^1$ map $\chi:\Omega\rightarrow \mathbb{R}^m$
 if $J_\chi(x_0):=\det f'(x_0)=0$.
A value $a\in\mathbb{R}^m$ is called a \emph{regular value}
  of $\chi$ 
  if $\chi^{-1}(a)$ contains no critical points of $\chi$; 
  otherwise it is a \emph{singular value} of $\chi$. 
A point $x_0$ is not a critical point of an analytic function $\chi$
 if and only if the algebraic multiplicity of $\chi$ at $x_0$ is one.

 \begin{lemma}
   \label{lem:orientationPreserved}
   An analytic function $\chi:\mathbb{C}\rightarrow\mathbb{C}$
   is locally orientation-preserving
   at any $z_0$ that is not a critical point of $\chi$,
   i.e., $J_\chi(z_0)>0$.
 \end{lemma}
 \begin{proof}
   It suffices to show that $\chi$ maps an infinitesimal circle
   $\varphi(\theta)=z_0 + r e^{\ii \theta}$
   to another infinitesimal circle $\psi(\theta) = a + R e^{\ii\theta}$
   so that the Jordan curves $\varphi([0,2\pi])$ and $\psi([0,2\pi])$ 
   have the same orientation.
   Since $\chi$ is analytic and $J_\chi\ne 0$, $R$ is a positive constant.
   Then we have, as $r\rightarrow 0$, 
   \begin{displaymath}
    \psi(\theta) - a = \chi(\varphi(\theta)) - \chi(z_0)
    = \chi'(z_0) [\varphi(\theta) - z_0] + O(r^2), 
  \end{displaymath}
  the derivative of which yields
  $[\psi(\theta) - a]'
  = \chi'(z_0) [\varphi(\theta) - z_0]' + O(r)$.
  Since the cross product of two planar vectors
  $\mathbf{u},\mathbf{v}$ 
  is given by
  $    \mathbf{u}\times\mathbf{v}
    := (0, 0, \det [\mathbf{u}, \mathbf{v}])^T$,
  we have $(A\mathbf{u})\times(A\mathbf{v})
  = (0,0,\det A \det[\mathbf{u}, \mathbf{v}])^T$
  for any matrix $A\in \mathbb{R}^{2\times 2}$.
  Therefore, the orientation of the circle $\psi$
  is related to that of $\varphi$ by 
  \begin{displaymath}
    [\psi(\theta) - a]\times [\psi(\theta) - a]'
    = (J_\chi(z_0)+O(r^2))[\varphi(\theta) - z_0]\times [\varphi(\theta) - z_0]',
  \end{displaymath}
  where $J_\chi(z_0)=\det \chi'(z_0)$. 
  By the Cauchy-Riemann equation, we have
  \begin{displaymath}
    \chi'(z_0) = 
    \begin{bmatrix}
      \alpha & -\beta \\  \beta & \alpha
    \end{bmatrix}
    \quad \Rightarrow\quad
    J_\chi(z_0)=\det \chi'(z_0) = \alpha^2+\beta^2 > 0, 
  \end{displaymath}
  which completes the proof.
 \end{proof}
  
\begin{corollary}
  \label{coro:windingNumberInTermsOfJf}
  For an oriented Jordan curve $\Gamma$, 
  an analytic map
  $\chi: \overline{\textrm{BCJ}(\Gamma)}\rightarrow\mathbb{C}$,
  and a regular value $a$ of $\chi$ satisfying $a\not\in \chi(\Gamma)$, 
  we have
  \begin{equation}
    \label{eq:windingNumberViaJacobian}
    w(\chi(\Gamma),a) = \sum_{z_j\in \chi^{-1}(a)} \sgn J_\chi(z_j).
  \end{equation}
\end{corollary}
\begin{proof}
  Lemma \ref{lem:orientationPreserved}
  gives $J_\chi(z_j)>0$,
  which further implies that $\sgn J_\chi(z_j)=+1$ or $-1$
  respectively for positively or negatively oriented $\Gamma$.
  Each $z_j$ has its algebraic multiplicity $m_j=1$.
  The rest follows from Theorem \ref{thm:argumentPrincipleAnalytic}
  and (\ref{eq:windingNumberJordan}). 
\end{proof}

For LFC through a static simple curve $\widetilde{LN}$ in $\mathbb{R}^2$, 
Zhang and Ding \cite{zhang2015generalized,zhang2019lagrangian}
 constructed a closed curve
 $\mathcal{G}_{\mathcal{D}}$
 from $\widetilde{LN}$, the velocity field $\mathbf{u}(\mathbf{x},t)$,
 and the time interval $(t_0,t_0+k)$,
 termed $\mathcal{G}_{\mathcal{D}}$ as the generating curve
 of donating regions,
 and defined donating regions
 as the equivalence classes of particles at $t_0$
 with respect to the winding numbers of $\mathcal{G}_{\mathcal{D}}$,
 i.e., 
 \begin{equation}
   \label{eq:DRviaGenCurves}
   \mathcal{D}_{\widetilde{LN}}^{n}(t_{0},k)
   :=\{\mathbf{p}(t_0)\;
   |\;w(\mathcal{G}_{\mathcal{D}},\mathbf{p}(t_0))=n\}.
 \end{equation} 
Using the Hopf theorem,
 they also showed the index-by-index equivalence
 of donating regions and flux sets. 
For LFC in three and higer dimensions,
this approach via winding numbers
clearly needs to be generalized.

\subsection{The topological degree}
\label{sec:topological-degree}

As a beautiful achievement of topology,
 the generalization
 of the winding number to the topological degree in $\mathbb{R}^m$
 spanned two centuries
 and involved many famous mathematicians
 such as Cauchy, Poincar\'{e},
 Brouwer, de Rham, and so on;
 see \cite[chap. 1]{outerelo2009mapping}
 for an excellent exposition on this history.
To make a long story short,
 we start from the axiomatization
 of three key features of winding numbers. 
 
\begin{theorem}
  \label{thm:topoDegreeRnExistence}
  There is at most one function $\deg: M\rightarrow\mathbb{Z}$, where 
  \begin{equation}
    \label{eq:topoDegreeRn}
    M:= \left\{(\chi,\Omega,y):
      \left\{
    \begin{array}{l}
      \Omega\subset \mathbb{R}^m \text{ open and bounded};
      \\
      \chi: \overline{\Omega} \rightarrow \mathbb{R}^m \text{ continuous};
      \\
      y\in \mathbb{R}^m \setminus \chi(\partial\Omega)
    \end{array}\right.
    \right\}, 
  \end{equation}
  that satisfies normalization, additivity, and homotopy, i.e., 
  \begin{enumerate}[(TPD-1)]
  \item $\deg(I,\Omega, y)=1$ for all $y\in\Omega$
    where $I$ is the identity map; 
  \item $\deg(\chi,\Omega,y) = \deg(\chi,\Omega_1,y) + \deg(\chi,\Omega_2,y)$
    if $\Omega_1$ and $\Omega_2$
    are disjoint open subsets of $\Omega$ such that
    $y\not\in \chi(\overline{\Omega}\setminus(\Omega_1\cup\Omega_2))$;
  \item $\deg(H(t,\cdot), \Omega, y(t))$
    is independent of $t\in[0,1]$ if both
    $H:[0,1]\times\overline{\Omega}\rightarrow \mathbb{R}^m$
    and $y:[0,1]\rightarrow \mathbb{R}^m$ are continuous
    and if $y(t)\not\in H(t,\partial\Omega)$ 
    for all $t\in[0,1]$. 
  \end{enumerate}
\end{theorem}
\begin{proof}
  See \cite[\S 1]{deimling10:_nonlin_funct_analy}.
\end{proof}


Such a function is constructed as follows.

\begin{definition}[Topological degree]
  \label{def:topoDegree}
  First, the \emph{topological degree} of $(\chi,\Omega, y_1)\in M$
  with $y_1$ being a regular value of $\chi\in {\mathcal C}^1(\Omega)$
  is given by
  \begin{equation}
    \label{eq:topoDegreeRegular}
    \deg(\chi,\Omega,y_1) := \sum_{x\in \chi^{-1}(y_1)} \sgn J_\chi(x), 
  \end{equation}
  where $J_\chi(x):=\det \chi'(x)$. 
  In particular, $\deg(\chi,\Omega,y)=0$ if $\chi^{-1}(y)=\emptyset$.

  Second, the \emph{topological degree} of $(g,\Omega, y)\in M$
  with $g\in {\mathcal C}^2(\Omega)$
  is defined as 
  \begin{equation}
    \label{eq:topoDegreeRegular2}
    \left\{
    \begin{array}{l}
      \deg(g,\Omega,y) := \deg(g,\Omega,y_1), \\
      |y_1-y|<\rho(y,g(\partial \Omega))
      := \min_{z\in g(\partial \Omega)} \|y-z\|_2, 
    \end{array}\right.
  \end{equation}
  where $y_1$ is a regular value of $g$
  and $\deg(g,\Omega,y_1)$ is given by (\ref{eq:topoDegreeRegular}).

  Finally, the \emph{topological degree} of $(\chi,\Omega, y)\in M$
  is defined as 
  \begin{equation}
    \label{eq:topoDegreeRegularFinal}
    \deg(\chi,\Omega,y) := \deg(g,\Omega,y), 
  \end{equation}
  where $g\in {\mathcal C}^2(\Omega)\cap {\mathcal C}\left(\overline{\Omega}\right)$
  is any map satisfying $\|g-\chi\|_{\infty}<\rho(y,\chi(\partial \Omega))$
  and $\deg(g,\Omega,y)$ is given by (\ref{eq:topoDegreeRegular2}).
\end{definition}

The first definition (\ref{eq:topoDegreeRegular})
clearly comes from Corollary \ref{coro:windingNumberInTermsOfJf}; 
 the second definition (\ref{eq:topoDegreeRegular2})
 is based on Sard's theorem
 that singular values form a set of measure zero; 
 the last definition (\ref{eq:topoDegreeRegularFinal})
 is reminiscent of Rouch\'{e}'s theorem in complex analysis.
Altogether, 
(\ref{eq:topoDegreeRegular}), (\ref{eq:topoDegreeRegular2}),
and (\ref{eq:topoDegreeRegularFinal}) form a sequence of well defined concepts 
that apply to the most general case of $\chi$ being merely continuous;
see \cite[\S 2]{deimling10:_nonlin_funct_analy}
for more details. 



\begin{theorem}[Product formula]
  \label{thm:productFormulaOfDegree}
  Suppose 
  $\chi\in {\mathcal C}\left(\overline{\Omega}\right)$
  where $\Omega\subset \mathbb{R}^m$ is open and bounded, 
  $g\in{\mathcal C}(\mathbb{R}^m)$,
  and $y\not\in (g\circ \chi)(\partial \Omega)$. 
  Then 
  \begin{equation}
    \label{eq:productFormulaOfDegree}
    \deg(g\circ \chi, \Omega, y)
    = \sum_i \deg(\chi,\Omega,K_i) \deg(g, K_i, y),
  \end{equation}
  where $K_i$'s are the bounded components
  of $\mathbb{R}^m\setminus \chi(\partial \Omega)$, 
  $\deg(\chi,\Omega,K_i)=\deg(\chi,\Omega,y_i)$
  for any $y_i\in K_i$, and 
  the summation has a finite number of nonzero terms.
\end{theorem}
\begin{proof}
  See \cite[\S 5]{deimling10:_nonlin_funct_analy}.
\end{proof}

The \emph{index of a continuous map}
 $\chi\in{\mathcal C}\left(\overline{\mathcal B}_{r_0}(x_0)\right)$
 at $x_0$ is defined as
 \begin{equation}
   \label{eq:indexOfContMapRn}
   j(\chi,x_0) := \deg(\chi, {\mathcal B}_r(x_0), \chi(x_0)),
 \end{equation}
 where ${\mathcal B}_r(x)$ is the open $n$-ball
 with its center at $x$
 and its radius $r$ sufficiently small
 such that $\chi(x)\ne \chi(x_0)$
 for all $x\in\overline{\mathcal B}_r(x_0)\setminus \{x_0\}$.
As a topological invariant,
 the index $j(\chi,x_0)$ characterizes
 the local behavior of $\chi$ at $x_0$:
 (TPD-3) in Theorem \ref{thm:topoDegreeRnExistence}
 dictates that
 $\deg(\chi, {\mathcal B}_r(x_0), \chi(x_0))
 = \deg(\chi, \tilde{{\mathcal B}}(x_0), \chi(x_0))$
 for any small neighborhood
 ${\mathcal B}_r(x_0)$ of $x_0$. 

\begin{lemma}
  \label{lem:indexOfAnalyticFuncViaProductFormula}
  For an analytic function
  $\chi: \overline{{\mathcal B}_r(x_0)}\rightarrow \mathbb{C}$
  satisfying $\chi(x)\ne \chi(x_0)=0$
  for any $x\in {\mathcal B}_r(x_0)\setminus\{x_0\}$,
  the index $j(\chi,x_0)$ in (\ref{eq:indexOfContMapRn})
  reduces to $m_0$,
  the algebraic multiplicity of $\chi$ at $x_0$.
\end{lemma}
\begin{proof}
  Construct a function $g(z) = x_0 + q z$ where $q\in(0,r)$. 
  Then we have
  \begin{displaymath}
    \begin{array}{rl}
      j(\chi, x_0) &= \deg(\chi, {\mathcal B}_q(x_0), 0)
                  = \deg(\chi\circ g, {\mathcal B}_1(0), 0)
      \\
                   & = w(\chi(g(\partial {\mathcal B}_1(0))),0 )
                     = w(\chi(\partial {\mathcal B}_q(0))),0)
                     = m_0,
    \end{array}
  \end{displaymath}
  where the first step follows from (\ref{eq:indexOfContMapRn}), 
  the second from Theorem \ref{thm:productFormulaOfDegree},
  the third from Definition \ref{def:topoDegree}
  and Corollary \ref{coro:windingNumberInTermsOfJf}, 
  and the last from Theorem \ref{thm:argumentPrincipleAnalytic}.
\end{proof}

As an alternate interpretation of
 the argument principle (\ref{eq:argumentPrincipleAnalytic}),
 the winding number $w(\chi(\Gamma),a)$
 is the number of preimages of $a$ under the analytic map $\chi$, 
 counted with algebraic multiplicities.
More generally,
 the topological degree $\deg(\chi,\Omega,y)$
 in Definition \ref{def:topoDegree}
 is the number of preimages of $y$
 under the continuous map $\chi$,
 counted with its indices in (\ref{eq:indexOfContMapRn}),
 i.e.,
 \begin{equation}
   \label{eq:arguPrincipleViaTopoDegree}
   \begin{array}{l}
     \deg(\chi,\Omega,y) = \sum_{z_j\in \chi^{-1}(y)} j(\chi, z_j).
   \end{array}
 \end{equation}

However, there is a prominent difference
 between analytic and continuous maps.
In the former case, 
 Theorem \ref{thm:argumentPrincipleAnalytic} furnishes an explicit
 algorithm for locating solutions of the equation
 $\chi(x)=a$: 
 draw a positively oriented Jordan curve $\Gamma$,
 map $\Gamma$ to the closed curve $\chi(\Gamma)$,
 and deduce from (\ref{eq:windingNumberJordan}) and 
 (\ref{eq:argumentPrincipleAnalytic}) that
 $w(\chi(\Gamma),a)$ equals the number of preimages of $a$
 in BCJ$(\Gamma)$, 
 counted with their algebraic multiplicity.
 
For a continuous map $\chi$,
the strongest statement of such nature is
\begin{equation}
  \label{eq:numberOfSolutionsCont}
  w(\chi(\Gamma),a)\ne 0 \ \Rightarrow\
  \chi^{-1}(a)\cap \textrm{BCJ}(\Gamma)\ne \emptyset.
\end{equation}

This weakening of (\ref{eq:argumentPrincipleAnalytic})
 and (\ref{eq:windingNumberJordan}) to (\ref{eq:numberOfSolutionsCont})
 is due to the fact that,
 in contrast to the algebraic multiplicity of an analytic map
 being always nonnegative,
 the topological multiplicity of a continuous map
 can be both positive and negative.
For example, the complex map $\chi(x+\ii y)=x+\ii |y|$
 with $x,y\in \mathbb{R}$
 is not analytic but continuous,
 and the equation $\chi(z) = a\ii$ with $a>0$
 has two solutions $z=\pm a\ii$. 
For a positively oriented
 Jordan curve $\Gamma$ with $\pm a\ii\in \text{BCJ}(\Gamma)$, 
 we have $w(\chi(\Gamma),a\ii)= 0$ 
 yet nonzero indices:
 $j(\chi, a\ii)=+1$ and $j(\chi,-a\ii)=-1$. 
 
\begin{corollary}
  \label{coro:TopoDegreeReduceToWindingNumbers}
  For a continuous function $\chi:\overline{\Omega}\rightarrow\mathbb{R}^2$
  that is ${\mathcal C}^1$ on $\Omega=\text{BCJ}(\Gamma)$, 
  its topological degree in Definition \ref{def:topoDegree}
  reduces to the winding number in (\ref{eq:windingNumber}). 
\end{corollary}
\begin{proof}
  By Cauchy's theorem,
  ${\mathcal C}^1$ complex functions are analytic.
  At each preimage $z_j$ of $y$,
  either $J_\chi(z_j)\ne 0$ or $J_\chi(z_j)= 0$.
  The former case is covered by the same form
  of (\ref{eq:windingNumberViaJacobian}) and
  (\ref{eq:topoDegreeRegular})
  while the latter case
  by Lemma \ref{lem:indexOfAnalyticFuncViaProductFormula}
  and (\ref{eq:arguPrincipleViaTopoDegree}). 
\end{proof}

Corollary \ref{coro:TopoDegreeReduceToWindingNumbers}
 also holds if $\chi$ is not ${\mathcal C}^1$
 but merely continuous. 
However, in this work, 
 Definition \ref{def:topoDegree} is only applied
 to the function $\chi$ in (\ref{eq:mapchi}), 
 which is also assumed to be ${\mathcal C}^1$.
Thus Corollary \ref{coro:TopoDegreeReduceToWindingNumbers}
 suffices to show that winding numbers
 are a special family of topological degrees in two dimensions. 
 
Unlike complex functions, 
 a ${\mathcal C}^1$ function $\phi:\mathbb{R}^m\rightarrow\mathbb{R}^m$
 is not automatically analytic for $m>2$, 
 in which case 
 Lemma \ref{lem:orientationPreserved} may not hold.
Fortunately, we show in Lemma \ref{lem:JacobiGreaterThanZero}
 that $J_{\phi}>0$ 
 if $\phi$ is the flow map of a ${\mathcal C}^1$ velocity
 in (\ref{eq:velODE}).  

\subsection{Immersion and sphere eversion}
\label{sec:sphere-eversion}

The \emph{immersion of a differentiable manifold
  ${\mathcal M}$ in $\mathbb{R}^m$}
 is a map $g: {\mathcal M}\rightarrow \mathbb{R}^m$
 such that at every $\mathbf{p}\in{\mathcal M}$
 its derivative
 $\dif g|_{\mathbf{p}} :
 T_{\mathbf{p}}{\mathcal M}\rightarrow T_{g(\mathbf{p})}\mathbb{R}^m$
 is an injective map, 
 where $T_{\mathbf{p}}{\mathcal M}$
 is the tangent space of ${\mathcal M}$ at $\mathbf{p}$.
Although $g$ needs not to be injective, 
 the implicit function theorem implies that 
 $g$ is locally a homeomorphism and thus a local embedding.
For example,
 any non-orientable closed surface such as the Klein bottle
 cannot be embedded in $\mathbb{R}^3$
 but can be immersed in $\mathbb{R}^3$. 
 
A \emph{regular homotopy
 between two immersions $g$ and $h$}
 from ${\mathcal M}$ to $\mathbb{R}^m$
 is a differentiable function
 $H: {\mathcal M}\times[0,1]\rightarrow \mathbb{R}^m$
 such that for every $t\in[0,1]$
 the function $H_t:{\mathcal M}\rightarrow \mathbb{R}^m$
 given by $H_t(\mathbf{x}):= H(\mathbf{x},t)$ is an immersion
 with $H_0=g$ and $H_1=h$.
Thus a regular homotopy
 is a homotopy of manifolds through immersions.

By the Whitney-Graustein theorem \cite{whitney37}, 
the regular homotopy classes of immersions
of the circle $\mathbb{S}^1$ in $\mathbb{R}^2$
are classified by the winding number. 
Thus
a differentiable closed curve with one orientation
is never regular homotopic
with another closed curve with the other orientation,
which, in the context of LFC, means that
a diffeomorphic flow map
never turns the Jordan curve inside out. 
Consequently,
 one can utilize this to simplify the matter of orienting a Jordan curve: 
 the outward normal vector can be determined
 \emph{once and for all} as pointing from the bounded complement
 to the unbounded complement.
This \emph{extrinsic} choice of the normal direction 
 always comply with the convention of LFC that
 flux be calculated from one side to the other.
In addition, 
 this extrinsic orientation simplifies LFC itself: 
 to cross the Jordan curve from one complement
 to the other twice, 
 a particle must return to its original complement, 
 but then this extrinsic orientation
 implies that the fluxing index be zeroed out 
 before the second crossing. 
Therefore, the only nonzero donating regions
 for a Jordan curve are 
 $\mathcal{D}_{\gamma}^{+1} (t_{0},k)$
 and $\mathcal{D}_{\gamma}^{-1} (t_{0},k)$
 and the Lagrangian flux reduces to
\cite[Corollary 10]{zhang2013highly}
\begin{displaymath}
  \begin{array}{l}
  \int_{t_0}^{t_0+k} \oint_{\gamma (t)}
  \chi(\mathbf{x},t) \bigl[
  \mathbf{u} (\mathbf{x},t) -\partial_{t}\gamma
  \bigr] \cdot \mathbf{n} (\mathbf{x},t) \dif \mathbf{x} \dif t
  =  \sum_{n=\pm 1}n\int_{\mathcal{D}_{\gamma}^{n} (t_{0},k)}
  f (\mathbf{x},t_{0})\,\dif \mathbf{x}.
  \end{array}
\end{displaymath}
In particular, provided that
two moving Jordan curves coincide both at $t_0$ and at $t_e$, 
their Lagrangian fluxes are the same.


As for ${\mathcal C}^2$ immersions
 of the sphere $\mathbb{S}^2$ in $\mathbb{R}^3$,
Smale \cite{smale1959classification} proved that
 any two such immesions are regularly homotopic. 
Thus there exists a sphere eversion,
 the process of turning a sphere inside out in $\mathbb{R}^3$
 without tearing or creasing on the sphere.
The first constructed example was exhibited by
Shapiro and Morin \cite{francis1980arnold};
see also the exquisite book and video
 by Levy and Thurston \cite{levy1995making}
 using the ``belt-trick.''
For recent developments of sphere eversion, 
 the reader is referred to \cite{litman2017sphere}. 

What is the consequence of sphere eversion on LFC 
 for a moving surface in $\mathbb{R}^3$?  
Clearly 
 extrinsic orientations of a closed surface
 via its complements of $\mathbb{R}^3$ 
 are no longer appropriate. 
 
As a simple counterexample,
 assume $\mathbf{u} (\mathbf{x},t) =\mathbf{0}$
 and $f(\mathbf{x},t) =1$.
Then the flow map is the identity 
$\phi_{t_0}^{+\tau k}(\mathbf{x})=\mathbf{x}$
for any $\tau\in[0,1]$. 
Consider two simple closed surfaces.
The first is the unit sphere fixed for all $t\in [t_{0},t_{0}+k]$;  
the corresponding flux is clearly zero.
The second is a sphere turned inside out once in $[t_{0},t_{0}+k]$
but coincide with the unit sphere both at $t_0$ and $t_0+k$;
in this case we have
\begin{displaymath}
  \begin{array}{l}
  \int_{t_0}^{t_0+k} \oint_{\mathcal{S}_{2} (t)}
  \chi(\mathbf{x},t) \bigl[
  \mathbf{u} (\mathbf{x},t) -\partial_{t}\mathcal{S}_{2}
  \bigr] \cdot \mathbf{n} (\mathbf{x},t)\, \dif \mathbf{x} \dif t
  =2 \int_{\mathcal{D}_{\mathcal{S}_{2}}^{2} (t_{0},k)}
  f (\mathbf{x},t_{0})\dif \mathbf{x}
  = \frac{8}{3}\pi,
  \end{array}
\end{displaymath}
where 
$\mathcal{D}_{\mathcal{S}_{2}}^{2} (t_{0},k)
=\left\{ \mathbf{x}\ |\ |\mathbf{x}|  < 1 \right\}$ is the unit ball. 

To sum up,
 the generalization of LFC from a fixed curve in $\mathbb{R}^2$
 to a moving hypersurface in $\mathbb{R}^m$
 is not straightforward,
 mostly because the switching of gears from winding numbers
 to topological degrees necessitates
 the tackling of a number of subtle issues
 that are covered up by the simple topology in two dimensions.
In particular, it is no longer adequate
 to use extrinsic orientations for closed surfaces.
The coordinate system on the moving surface
 must be oriented \emph{intrinsically}
 from the parametrization of the surface.



\section{Analysis}
\label{sec:analysis}

In this section, 
we give intrinsic orientations to hypersurfaces and cycles,
characterize fluxing sets by topological degrees, 
derive integration formulas on cycles, 
and prove a flux identity
that is best suited for numerical LFC algorithms.



\subsection{Orienting hypersurfaces and cycles}
\label{sec:Hypersurf}

Hereafter we denote by $\mathbb{B}^m:=(0,1)^{m}$ the open $m$-cube. 

\begin{definition}
  \label{def:ParameterizedSurf}
  A \emph{(parameterized) hypersurface}
  or a \emph{(spherical) cycle}
  is the image of a continuous map 
  $\mathcal{S}: \Omega\rightarrow\mathbb{R}^{m}$
  where $\Omega=\mathbb{B}^{m-1}$ or $\Omega=\partial \mathbb{B}^m$,
  respectively.
%
\end{definition}

A hypersurface or cycle is \emph{simple}
 if ${\mathcal S}$ is injective
 and it is \emph{regular} if 
 $\mathcal{S}\in {\mathcal C}^{1}(\Omega)$ 
 and 
 $\text{rank}(\dif\mathcal{S}(\mathbf{z}))=m-1$
 for all $\mathbf{z}\in\mathbb{B}^{m-1}$.

\begin{definition}
  A \emph{moving hypersurface} 
  is a homotopy class
  $\mathcal{S}: 
  \mathbb{B}^{m-1}\times [0,1]\rightarrow\mathbb{R}^{m}$
  of simple regular hypersurfaces,
  each of which is homeomorphic to ${\mathcal S}(0)$.
\end{definition}

We write $\mathcal{S}(t):=\{\mathcal{S}(\mathbf{z},t)\;
 |\;\mathbf{z}\in\mathbb{B}^{m-1}\}$
 for the point set of a moving hypersurface
 at a fixed time $t$. 
To emphasize the parametrization and points on the hypersurface, 
 we write
 $\mathcal{S}(\mathbf{z})=
 (p_{1}(\mathbf{z}),\ldots,p_{m}(\mathbf{z}))$
 or $\mathcal{S}(\mathbf{z}, t)=
 (p_{1}(\mathbf{z},t),\ldots,p_{m}(\mathbf{z},t))$.

By the word ``spherical,'' we recall 
 that a topological $m$-cycle may not be homeomorphic
 to $\partial \mathbb{B}^m$ or $\mathbb{S}^{m-1}$. 
In this work, however, 
 a cycle always refers to a spherical cycle;
 thus for simplicity we drop the word ``spherical.''
We also assume that
 all hypersurfaces be regular, 
 this assumption incurs no loss of generality for LFC
 because any hypersurfaces with discontinuous or degenerate derivatives 
 can be approximated to arbitrary accuracy 
 by a regular hypersurface.

\begin{definition}
  \label{def:OutwardNvecForSurf}
  The \emph{outward normal vector of a regular hypersurface} 
  \mbox{$\mathcal{S}: \mathbb{B}^{m-1}\rightarrow \mathbb{R}^m$}
  at $\mathcal{S}(\textbf{z})$
  is the unit vector $\mathbf{n}(\mathbf{z})$ satisfying
  \begin{equation}
    \label{eq:OutwardVecOfMovingSurf}
    \left\{
      \begin{array}{l}
        \left(\bigwedge_{i=1}^{m-1}
          \frac{\partial\mathcal{S}(\mathbf{z})}{\partial z_{i}}
        \right)
        \wedge\mathbf{n}(\mathbf{z})<0;
        \\
        \forall j=1,\ldots,m-1,\quad
        \pdfFrac{\mathcal{S}(\mathbf{z})}{z_{j}}
        \cdot\mathbf{n}(\mathbf{z}) = 0.
      \end{array}
    \right.
  \end{equation}
  where the parameter $\mathbf{z}:=(z_{1},\ldots,z_{m-1})$
  and $\wedge$ denotes the wedge product \cite{tu11:_introd_manif}.
\end{definition}

\begin{definition}
  \label{def:OutwardNormalVec}
  The \emph{outward normal vector of a cycle} 
  $\psi: \partial \mathbb{B}^m\rightarrow \mathbb{R}^m$
  at $\psi(\mathbf{z})$
  is the unit vector $\mathbf{n}(\mathbf{z})$ satisfying
  \begin{itemize}
  \item for $\mathbf{z}_{i}^{0} :=
    (z_{1},\ldots,z_{i-1},0,z_{i+1},\ldots,z_{m})$,
    \begin{equation}
      \label{eq:OutwardVecOfClosedSurf1}
      \left\{
        \begin{array}{l}
          \left(\bigwedge_{k=1}^{i-1}
          \pdfFrac{\psi(\mathbf{z}_{i}^{0})}{z_{k}}\right)
          \wedge\mathbf{n}(\mathbf{z}_{i}^{0})\wedge
          \left(\bigwedge_{k=i+1}^{m}
          \pdfFrac{\psi(\mathbf{z}_{i}^{0})}{z_{k}}\right)\le0,
          \\
          \forall k=1,\ldots,i-1,i+1,\ldots,m,\quad 
          \pdfFrac{\psi(\mathbf{z}_{i}^{0})}{z_{k}}
          \cdot\mathbf{n}(\mathbf{z}_{i}^{0})=0; 
        \end{array}
      \right.
    \end{equation}
  \item for $\mathbf{z}_{i}^{1}:=
    (z_{1},\ldots,z_{i-1},1,z_{i+1},\ldots,z_{m})$,
    \begin{equation}
      \label{eq:OutwardVecOfClosedSurf2}
      \left\{
        \begin{array}{l}
          \left(\bigwedge_{k=1}^{i-1}
          \pdfFrac{\psi(\mathbf{z}_{i}^{1})}{z_{k}}\right)
          \wedge\mathbf{n}(\mathbf{z}_{i}^{1})
          \wedge\left(\bigwedge_{k=i+1}^{m}
          \pdfFrac{\psi(\mathbf{z}_{i}^{1})}{z_{k}}\right)\ge0,
          \\
          \forall k=1,\ldots,i-1,i+1,\ldots,m, \quad 
          \pdfFrac{\psi(\mathbf{z}_{i}^{1})}{z_{k}}
          \cdot\mathbf{n}(\mathbf{z}_{i}^{1})=0.
        \end{array}
      \right.
    \end{equation}
  \end{itemize}
\end{definition}

Definitions \ref{def:OutwardNvecForSurf}
 and \ref{def:OutwardNormalVec}
 give intrinsic orientations
 since the direction of an outward normal vector 
 is determined by the parametrization. 

Instead of ``$< 0$'' and ``$> 0$'',
 we write ``$\le 0$'' and ``$\ge 0$''
 in (\ref{eq:OutwardVecOfClosedSurf1})
 and (\ref{eq:OutwardVecOfClosedSurf2})
 to indicate that the cycle $\psi(\partial \mathbb{B}^{m})$
 may contain \emph{singular point}s
 where $\text{rank}(\dif\psi|_{\mathbf{z}_{i}^{s}})=0$.
Fortunately, Sard's theorem implies that 
 the singular points on $\psi(\partial \mathbb{B}^{m})$ 
 form a set of measure zero
 and thus their presence does not affect integrals over cycles.
 

\subsection{The fluxing index and flux sets}
\label{sec:FluxingIndex}

In this subsection, we define the fluxing index precisely 
 and show that it is the topological degree of some function
 related to the flow map and the moving hypersurface.

\begin{definition}[Particle crossings through a hypersurface]
  \label{def:crossings}
  Suppose a Lagrangian particle $\mathbf{p}$ goes into 
  the moving hypersurface ${\mathcal S}(t)$
  at $t_{\times}:=t_0+\tau k$, 
  i.e.,
  \begin{displaymath}
    \mathbf{p}(t_{\times}):=\phi_{t_{0}}^{+\tau k}(\mathbf{p})
    =\mathcal{S}(\mathbf{\mathbf{z}}_{\mathbf{p}}, t_{\times}), 
  \end{displaymath}
  where $\mathbf{\mathbf{z}}_{\mathbf{p}}$
  is the parameter of $\mathbf{p}(t_{\times})$
  on ${\mathcal S}(t_{\times})$.
  According to the relative velocity
  \begin{equation}
    \label{eq:relativeCrossVel}
    \mathbf{v}_{\times}(\mathbf{p},\tau)
    := \mathbf{u}(\mathbf{p}(t_{\times}), t_{\times})
    - \partial_t{\mathcal S}(\mathbf{z}_{\mathbf{p}}, t_{\times}), 
  \end{equation}
  the intersection $\mathbf{p}(t_{\times})$
  is called a \emph{positive crossing},
  a \emph{negative crossing}, or an \emph{improper intersection}
  if $\mathbf{v}_{\times}(\mathbf{p},\tau) \cdot
  \mathbf{n}(\mathbf{\mathbf{z}}_{\mathbf{p}}, t_{\times})$
  is positive, negative, or zero, respectively.
\end{definition}
 
\begin{definition}
  \label{def:fluxingIndex}
  The \emph{fluxing index of Lagrangian particle} $\mathbf{p}$
  passively advected
  by the flow of a time-dependent velocity field $\mathbf{u}$
  through a moving hypersurface ${\mathcal S}(t)$
  within a time interval $(t_0, t_0+k)$
  is the integer
  $n_{\mathbf{p}}(\mathbf{u}, t_0, k, {\mathcal S}):=n_{+}-n_{-}$
  where $n_{+}$ and $n_{-}$
  are respectively the numbers of its positive crossings
  and its negative crossings through ${\mathcal S}(t)$
  within $(t_0, t_0+k)$.
\end{definition}

By Definition \ref{def:crossings},
 the above fluxing index can be expressed as
\begin{equation}
  \label{eq:ExpressionForFluxIndices}
  n_{\mathbf{p}}(\mathbf{u}, t_0, k, {\mathcal S})
  = \sum_{\tau\in T_{\times}}\mathrm{sign}
  [\mathbf{v}_{\times}(\mathbf{p},\tau)\cdot
  \mathbf{n}(\mathbf{z}_{\mathbf{p}},t_{0}+\tau k)],
\end{equation}
where $T_{\times}:=\left\{\tau\in(0,1)\ |\
\phi_{t_0}^{+\tau k}(\mathbf{p})\in\mathcal{S}(t_{0}+\tau k)\right\}$; 
see Figure \ref{fig:fluxingIndex} for an illustration.

\begin{figure}
  \centering
  \subfigure[$n_{\mathbf{p}}=-1$]{
    \includegraphics[width=0.3\linewidth]{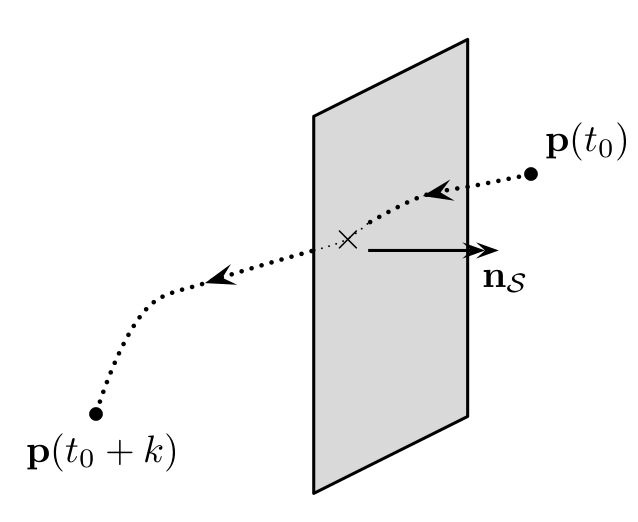}
  }
  \hfill
  \subfigure[$n_{\mathbf{p}}=0$]{
    \includegraphics[width=0.3\linewidth]{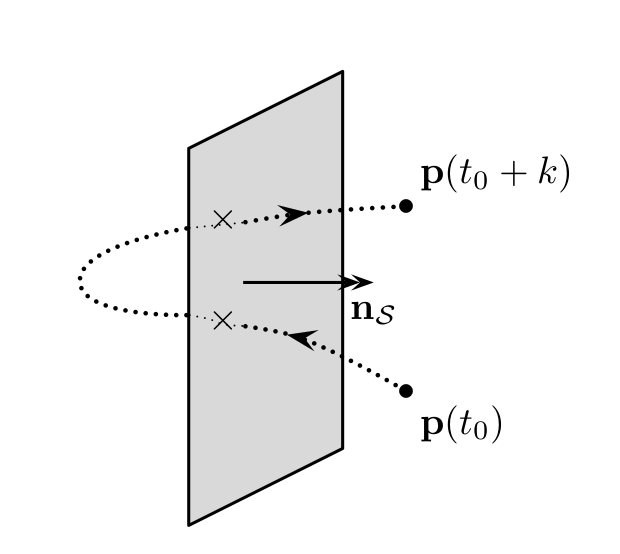}
  }
  \hfill
  \subfigure[$n_{\mathbf{p}}=+2$]{
    \includegraphics[width=0.3\linewidth]{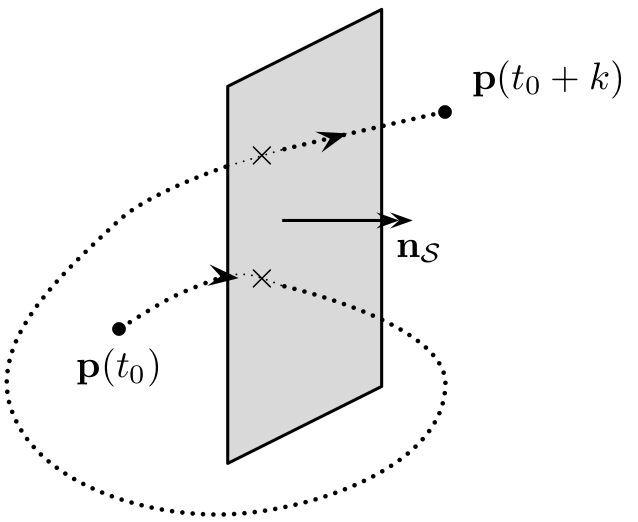}
  }
  \setcounter{subfigure}{0}
  \caption{Fluxing indices $n_{\mathbf{p}}$
    of a Lagrangian particle $\mathbf{p}$
    through a static surface $\mathcal{S}$. 
  A shaded region represents $\mathcal{S}$, 
  a dotted line the pathline $\Phi_{t_0}^{+k}(\mathbf{p})$, 
  and a marker ``$\times$'' a crossing point
  $\Phi_{t_0}^{+k}(\mathbf{p}) \cap \mathcal{S}$.}
  \label{fig:fluxingIndex}
\end{figure}


It is difficult to calculate the fluxing index
 by (\ref{eq:ExpressionForFluxIndices})
 or Definition \ref{def:fluxingIndex}.
Instead, 
 we link Definition \ref{def:fluxingIndex}
 to the flow map and the parametrization of the hypersurface.

\begin{theorem}
  \label{thm:FluxWithPsi}
  Suppose a Lagrangian particle $\mathbf{p}$ 
  crosses a moving hypersurface $\mathcal{S}(t)$
  at time $t_{\times}:=t_{0}+\tau k$.
  Then we have 
  \begin{equation}
      \label{eq:FluxSetWithDet}
      \sgn\left\{{\mathbf{v}_{\times}(\mathbf{p},\tau)}
      \cdot\mathbf{n}(\mathbf{z}_{\mathbf{p}},t_{\times})\right\}
      =\sgn\left\{\det\dif\phi_{t_0}^{+\tau k}(\mathbf{p})
        \det\dif\chi(\mathbf{z}_{\mathbf{p}},\tau)\right\}, 
  \end{equation}
  where $\mathbf{\mathbf{z}}_{\mathbf{p}}$
  is the parameter of the crossing point $\mathbf{p}(t_{\times})$
  on ${\mathcal S}(t_{\times})$,
  $\mathbf{n}(\mathbf{z}_{\mathbf{p}},t_{\times})$
  the unit outward normal vector of ${\mathcal S}(t_{\times})$
  at $\mathbf{p}(t_{\times})$,
  and
  the composite map $\chi:\mathbb{B}^m\rightarrow\mathbb{R}^{m}$ given by
  \begin{equation}
    \label{eq:mapchi}
    \chi(\mathbf{z},\tau) :=
    \phi_{t_{0}+\tau k}^{-\tau k}(\mathcal{S}(\mathbf{z},t_{0}+\tau k)). 
  \end{equation}
\end{theorem}
\begin{proof}
  Since $\tau, t_{\times}, \mathbf{p}(t_{\times}), \mathbf{z}_{\mathbf{p}}$
  are all constants in the proof,
  we write $\mathbf{e}_{i} :=
  \partial_{z_{i}} \mathcal{S}(\mathbf{z}_{\mathbf{p}},t_{\times})$,
  $\mathbf{n}_{\times}:=\mathbf{n}(\mathbf{z}_{\mathbf{p}},t_{\times})$,
  and $\mathbf{v}_{\times}:= \mathbf{v}_{\times}(\mathbf{p},\tau)$. 
  Then $(\mathbf{e}_i)_{i=1}^{m-1}$ 
  is a basis of the tangent space of 
  $\mathcal{S}(t_{\times})$
  at $\mathbf{p}(t_{\times})$.
  By Definition \ref{def:OutwardNvecForSurf},
  the normal vector satisfies
  \begin{equation}
    \label{eq:NormalVec}
    \wedge_{i=1}^{m-1}\mathbf{e}_{i}
    \wedge \mathbf{n}_{\times}<0.
  \end{equation}

  Write $\mathbf{x}:=\mathbf{p}(t_{\times})
  =\mathcal{S}(\mathbf{z}_{\mathbf{p}},t_{\times})$
  and we have, from (\ref{eq:PhiExpression}) and (\ref{eq:mapchi}), 
  \begin{equation}
      \label{eq:PhiAndInvPhi}
      \phi_{t_{0}}^{+\tau k}(\phi_{t_{\times}}^{-\tau k}(\mathbf{x}))
      = \mathbf{x} = 
      \phi_{t_{0}}^{+\tau k}(\chi(\mathbf{z}_{\mathbf{p}},\tau)). 
  \end{equation}

  Differentiate the first equality in \eqref{eq:PhiAndInvPhi},
  apply the chain rule, and we have
  \begin{displaymath}
    \difFrac{\phi_{t_{0}}^{+\tau k}(\phi_{t_{\times}}^{-\tau k}
      (\mathbf{x}))}{\tau}=\difFrac{\mathbf{x}}{\tau}
    \ \Rightarrow\ 
    \pdfFrac{\phi_{t_{0}}^{+\tau k}(\mathbf{p})}{\tau}
    +\dif\phi_{t_{0}}^{+\tau k}(\mathbf{p})
    \partial_{\tau}\phi_{t_{\times}}^{-\tau k}(\mathbf{x})=k
    \partial_{t}\mathcal{S}(\mathbf{z}_{\mathbf{p}},t_{\times}), 
  \end{displaymath}
  which, together with
  $\chi(\mathbf{z}_{\mathbf{p}},\tau)=
  \phi_{t_{\times}}^{-\tau k}(\mathbf{x})$
  and (\ref{eq:relativeCrossVel}), gives
  \begin{equation}
    \label{eq:DiffByt}
    \dif\phi_{t_0}^{+\tau k}(\mathbf{p})
    \partial_{\tau}\chi(\mathbf{z}_{\mathbf{p}},\tau)
    =-k\mathbf{v}_{\times}. 
  \end{equation}

  Differentiate the second equality in \eqref{eq:PhiAndInvPhi}
  and we have
  \begin{equation}
    \label{eq:DiffByTangentDirec}
    \dif \phi_{t_0}^{+\tau k}(\mathbf{p})
    \partial_{z_{i}}\chi(\mathbf{z}_{\mathbf{p}},\tau)=
    \pdfFrac{\mathbf{x}}{z_{i}} = \mathbf{e}_{i}.
  \end{equation}

  \eqref{eq:DiffByt} and \eqref{eq:DiffByTangentDirec} combine to
      $\dif \phi_{t_{0}}^{+\tau k}(\mathbf{p})
      \dif \chi(\mathbf{z}_{\mathbf{p}},\tau)
      =[\mathbf{e}_{1},\ldots,\mathbf{e}_{m-1}, 
      -k\mathbf{v}_{\times}]$,
  where the LHS is a matrix with column vectors 
  $\mathbf{e}_{i}$'s and
  $-k\mathbf{v}_{\times}$.
  Then properties of determinants and wedge products yield
  \begin{displaymath}
    \det\dif\phi_{t_0}^{+\tau k}(\mathbf{p})
      \det\dif\chi(\mathbf{z}_{\mathbf{p}},\tau)
      =-k(\wedge_{i=1}^{m-1}
        \mathbf{e}_{i})
        \wedge \mathbf{v}_{\times}
      = -k[\mathbf{v}_{\times}
      \cdot\mathbf{n}_{\times}]
      (\wedge_{i=1}^{m-1}\mathbf{e}_{i})
         \wedge\mathbf{n}_{\times}
  \end{displaymath}
  and the proof is completed by (\ref{eq:NormalVec}). 
\end{proof}

The map $\chi$ in (\ref{eq:mapchi}) is not a homeomorphism
 because, although ${\mathcal S}$ is injective, 
 $\phi_{t_{0}+\tau k}^{-\tau k}$ is not:
 a particle may visit the same location
 at two different time instants.


\begin{lemma}[Jacobi's formula]
  \label{lem:JacobiFormula}
  For a $\mathcal{C}^1$ velocity $\mathbf{u}$ in (\ref{eq:velODE}), 
  the Jacobian determinant
  $J$ of the flow map
  $\phi_{t_0}: \mathbb{R}^{m}\times\mathbb{R}\rightarrow\mathbb{R}^{m}$
  with fixed initial time $t_0$ satisfies
  \begin{equation}
    \label{eq:JacobiEquation}
    \frac{\dif J(\mathbf{p}(t_0),t)}{\dif t} 
    = J(\mathbf{p}(t_0), t)\, 
    \nabla \cdot \mathbf{u}(\mathbf{p}(t), t), 
  \end{equation}
  where the divergence operator $\nabla \cdot$
  only operates on spatial coordinates. 
\end{lemma}
\begin{proof}
  See \cite[p. 8]{chorin1993mathematical}.
\end{proof}

\begin{lemma}
  \label{lem:JacobiGreaterThanZero}
  The flow map
  of a $\mathcal{C}^1$ velocity $\mathbf{u}$ in (\ref{eq:velODE})
  preserves orientations, i.e., 
  \begin{equation}
    \label{eq:positiveDetJacobian}
    \forall\mathbf{p}(t_0)\in\mathbb{R}^{m},
    \forall k>0, \forall \tau\in(0,1),
    \quad
    J(\mathbf{p}(t_0),\tau)
    :=\det\dif \phi_{t_{0}}^{+\tau k}(\mathbf{p})> 0.
  \end{equation}
\end{lemma}
\begin{proof}
  By (\ref{eq:PhiExpression}), we have
  $\phi_{t_{0}}^{0}(\mathbf{p})=\mathbf{p}$ and 
  $\frac{\dif}{\dif \tau} \phi_{t_0}^{+\tau k}(\mathbf{p})
  = k\mathbf{u}\left(\mathbf{p}(t_{0}+\tau k), t_{0}+\tau k\right)$.
  Then Lemma \ref{lem:JacobiFormula} yields
  an ODE on $J(\tau):=J(\mathbf{p}(t_0),\tau)$, 
  \begin{equation}
    \label{eq:ODEonJ}
    \difFrac{J(\tau)}{\tau} = R(J, \tau)
    := J(\tau)\, k\, 
    \nabla \cdot \mathbf{u}(\mathbf{p}(t_0+\tau k), t_{0}+\tau k), 
  \end{equation}
  where the initial condition is $J(0)=1$.

  Now suppose there exists $\tau_{*}\in(0,1)$ such that 
  $J(\tau_{*})=0$.
  Since $\mathbf{u}\in {\mathcal C}^1$,
  $\nabla\cdot\mathbf{u}$ exists and is continuous; 
  thus $k\nabla\cdot\mathbf{u}$ is bounded on $[0,1]$.
  Therefore,
  $R(J, \tau)$ is Lipschitz continuous in $J$
  and continuous in $\tau$.
  By the Cauchy-Lipschitz theorem,
  there exists some $\epsilon>0$ such that 
  the ODE (\ref{eq:ODEonJ}) admits a unique solution $J(\tau)$
  on $[\tau_{*}-\epsilon, \tau_{*}+\epsilon]$, 
  which must be $J(\tau)\equiv 0$.
  But this contradicts the initial condition.
\end{proof}


The identity \eqref{eq:ExpressionForFluxIndices}, 
  Theorem \ref{thm:FluxWithPsi},
  Lemma \ref{lem:JacobiGreaterThanZero}, 
  and Definition \ref{def:topoDegree} yield

\begin{theorem}
  \label{thm:fluxingIndexAsMapDegree}
  The fluxing index of a Lagrangian particle 
  in Definition \ref{def:fluxingIndex} can be expressed as
  \begin{equation}
    \label{eq:fluxingIndexAsMapDegree}
    n_{\mathbf{p}}(\mathbf{u}, t_0, k, {\mathcal S})
    = \deg ( \chi, \mathbb{B}^m,\mathbf{p}).    
  \end{equation}
\end{theorem}

\begin{definition}
  \label{def:FluxSet}
  The \emph{flux set of index $n$}
  through a moving hypersurface $\mathcal{S}(t)$
  within a time interval $(t_0, t_0+k)$, 
  denoted $\mathcal{F}_{\mathcal{S}}^{n}(t_{0},k)$, 
  is the set of initial loci of all Lagrangian particles
  with fluxing index $n$.
\end{definition}

\subsection{The generating cycle}

In contrast to the implicit characterization of flux sets
in Theorem \ref{thm:fluxingIndexAsMapDegree}, 
the following is an explicit construction. 

\begin{definition}
  \label{def:generatingCycle}
  The \emph{generating cycle 
  of a moving hypersurface} $\mathcal{S}(t)$
  in the flow of a $\mathcal{C}^1$ velocity field 
  $\mathbf{u}(\mathbf{x},t)$ 
  over a time interval $(t_{0},t_{0}+k)$ 
  is 
  \begin{equation}
    \label{eq:FluxGeneratingHypersurface}
    \mathcal{G}_{\mathcal{D}}(t_{0},k) :=
    \mathcal{S}(t_{0}) \bigcup
    \phi_{t_0+k}^{-k}\left(\mathcal{S}(t_{0}+ k)\right)
    \bigcup \Psi_{\partial\mathcal{S}}(t_{0},k),
  \end{equation}
  where the \emph{streak hypersurface} seeded
  from $\mathcal{S}(\partial \mathbb{B}^{m-1},t_0)$
  is
  \begin{equation}
    \label{eq:streakHypersurface}
    \Psi_{\partial\mathcal{S}}(t_{0},k) :=
    \left\{\phi_{t_{0}+\tau k}^{-\tau k}(\mathbf{x})\;|\;
      \mathbf{x}=\mathcal{S}(\mathbf{z},t_{0}+\tau k),
      \mathbf{z}\in\partial \mathbb{B}^{m-1},\tau\in[0,1]\right\}.
  \end{equation}
\end{definition}

In two dimensions,
 $\Psi_{\partial\mathcal{S}}(t_{0},k)$
 and $\mathcal{G}_{\mathcal{D}}(t_{0},k)$
 are respectively the streaklines
 and the generating curve of a donating region;
 see \cite[Fig. 3.1 \& 4.1]{zhang2013donating}. 

\begin{lemma}
  \label{lem:generatingCycleAsImage}
  The generating cycle in Definition \ref{def:generatingCycle}
  is orientable and
  \begin{equation}
    \label{eq:generatingCycleAsImage}
    \mathcal{G}_{\mathcal{D}}(t_{0},k)=\chi(\partial \mathbb{B}^m).
  \end{equation}
  Furthermore, there exists 
  a parametrization and an orientation of
  $\mathcal{G}_{\mathcal{D}}(t_{0},k)$
  such that the normal vector of $\mathcal{S}(t_{0})$
  is determined by (\ref{eq:OutwardVecOfClosedSurf1})
  with $i=m$ (and is thus the same as
  that by Definition \ref{def:OutwardNvecForSurf}),
  that of $\phi_{t_0+k}^{-k}(\mathcal{S}(t_{0}+k))$
  by (\ref{eq:OutwardVecOfClosedSurf2}) with $i=m$,
  and that of $\Psi_{\partial\mathcal{S}}(t_{0},k)$
  by Definition \ref{def:OutwardNormalVec} 
  with $i<m$.
\end{lemma}
\begin{proof}
  (\ref{eq:streakHypersurface}) and (\ref{eq:mapchi}) yield
  $\Psi_{\partial\mathcal{S}}(t_{0},k)
  =\chi\left(\partial \mathbb{B}^m\setminus\partial \mathbb{B}^{m-1}\times\{0,1\}\right)$. 
  Then (\ref{eq:generatingCycleAsImage})
  follows from (\ref{eq:FluxGeneratingHypersurface}) and
  \begin{displaymath}
   \mathcal{S}(t_{0})\bigcup
  \phi_{t_0+k}^{-k}\left(\mathcal{S}(t_{0}+ k)\right) = 
  \left\{\phi_{t_{0}+\tau k}^{-\tau k}(\mathbf{x})\ |\ 
    \mathbf{x}\in \mathcal{S}(t_{0}+\tau k),\tau\in\{0,1\}\right\}.
  \end{displaymath}
  The above arguments imply that $\chi$ is a parametrization of
  $\mathcal{G}_{\mathcal{D}}(t_{0},k)$.

  The orientability of $\mathcal{G}_{\mathcal{D}}(t_{0},k)$
  follows from that of $\mathbb{B}^m$
  and $\chi$ being continuous.
  Then the proof is completed by orienting $\mathbb{B}^m$
  according to Definition \ref{def:OutwardNormalVec}, 
  selecting $\tau$
  as the $m$th coordinate of $\mathbb{B}^m$, 
  and choosing $\mathcal{S}(t_{0})$
  as the image of the lower $(m-1)$-face of $\mathbb{B}^m$
  normal to the $\tau$ axis.
\end{proof}
  
The above concepts are exemplified
in Figure \ref{fig:ExampleGeneratingSurface} where
$t_0=0$, $k=1$, and 
 \begin{displaymath}
   \begin{array}{rl}
     \chi(z_{1},z_{2},\tau) &= \phi_{\tau}^{-\tau}(z_{1}-\tau,z_{2},0)
                           =(z_{1}-\tau,z_{2},\tau),
     \\
     \Psi_{\partial {\mathcal S}} &= \cup_{i=1}^{4}\mathcal{Q}_{i},
     \\
     \mathcal{G}_{\mathcal{D}}(0,1) &=\chi(\partial \mathbb{B}^{3})
      = \mathcal{S}(0)\cup \phi_1^{-1}(\mathcal{S}(1))
      \cup \Psi_{\partial {\mathcal S}}, 
   \end{array}
 \end{displaymath}
 with the constituting parallelograms as 
  \begin{equation}
    \label{eq:partialDSFourCurves}
    \begin{array}{rl}
      \mathcal{Q}_1&=\chi(\{0\}\times[0,1]\times[0,1])
      =\left\{(-t,z_{2},t)\;|\;z_{2},t\in [0,1]\right\},\\
      \mathcal{Q}_{2}&=\chi(\{1\}\times[0,1]\times[0,1])
      =\left\{(1-t,z_{2},t)\;|\;z_{2},t\in[0,1]\right\},\\
      \mathcal{Q}_{3}&=\chi([0,1]\times\{0\}\times[0,1])
      =\left\{(z_{1}-t,0,t)\;|\;z_{1},t\in[0,1]\right\},\\
      \mathcal{Q}_{4}&=\chi([0,1]\times\{1\}\times[0,1])
      =\left\{(z_{1}-t,1,t)\;|\;z_{1},t\in[0,1]\right\};\\
      \mathcal{S}(0)&=\chi([0,1]\times[0,1]\times\{0\})
      =\left\{(z_{1},z_{2},0)\;|\;z_{1},z_{2}\in(0,1)\right\},\\
      \phi_1^{-1}(\mathcal{S}(1))&=\chi([0,1]\times[0,1]\times\{1\})
      =\left\{(z_{1}-1,z_{2},1)\;|\;z_{1},z_{2}\in(0,1)\right\}.
    \end{array}
  \end{equation}

\begin{figure}
  \centering
  \includegraphics[width=0.5\textwidth]{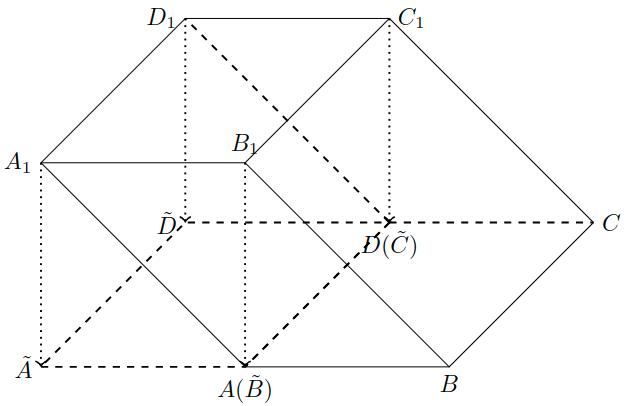}
  \caption{The generating cycle $\mathcal{G}_{\mathcal{D}}$
    of a moving square
    $\mathcal{S}(z_{1},z_{2},t)=(z_{1}-t,z_{2},0)$
    in the flow of $\mathbf{u}(x,y,z,t)\equiv(0,0,-1)$ 
    over the time interval $(0,1)$. 
    The squares $ABCD$, 
    $\tilde{A}\tilde{B}\tilde{C}\tilde{D}$, 
    and $A_1B_1C_1D_1$ 
    represent $\mathcal{S}(0)$, $\mathcal{S}(1)$,
    and $\phi_1^{-1}(\mathcal{S}(1))$,
    respectively.
    The boundary of the parallelepiped 
    $ABCD-A_1B_1C_1D_1$ constitutes $\mathcal{G}_{\mathcal{D}}$. 
    The four dotted lines are the pathlines of 
    $A_1, B_1, C_1$ and $D_1$. 
  }
  \label{fig:ExampleGeneratingSurface}
\end{figure}

  Note that we have parametrized $\mathcal{G}_{\mathcal{D}}(0,1)$
  such that $\mathcal{S}(0)$ and $\phi_1^{-1}(\mathcal{S}(1))$
  are images of $\mathbb{B}^2$ for the low side $\tau=0$
  and the high side $\tau=1$, respectively. 
  By Definition 
  \ref{def:OutwardNormalVec}, 
  the outward normal vectors of $\mathcal{G}_{\mathcal{D}}$ are:
  \begin{equation}
    \label{eq:ExampleOutwardNormalVec}
    \begin{array}{c}
      \mathbf{n}_{\mathcal{Q}_1}
      =\frac{\sqrt{2}}{2}\left(-1,0,1\right),\, 
      \mathbf{n}_{\mathcal{Q}_2}
      =\frac{\sqrt{2}}{2}\left(1,0,1\right),\,
      \mathbf{n}_{\mathcal{Q}_3}=\left(0,-1,0\right),\, 
      \mathbf{n}_{\mathcal{Q}_4}=\left(0,1,0\right); 
      \\
      \mathbf{n}_{\mathcal{S}(0)}=(0,0,-1),\ 
      \mathbf{n}_{\phi_1^{-1}(\mathcal{S}(1))} =\left(0,0,1\right).
    \end{array}
  \end{equation}

\subsection{The divergence theorem
  and Reynolds transport theorem for cycles}
\label{sec:diverg-Reynolds-theorems}

In this subsection, we customize the divergence theorem
 and Reynolds transport theorem for generating cycles
 that has the special structure
 of a topological sphere $\partial\mathbb{B}^m$.
Compared to the classical versions of these theorems, 
 our versions are less general 
 in that the proofs depend on the simple structure of $m$-cubes, 
 but are more general
 in that they are also valid on self-intersecting hypersurfaces. 

\begin{theorem}
  \label{thm:div}
  Define a cycle $\mathcal{S}:=\varphi(\partial \mathbb{B}^m)$
  with $\varphi: \overline{\mathbb{B}^m}\rightarrow \mathbb{R}^m$
  being a ${\mathcal C}^1$ map that needs not be injective.
  For a ${\mathcal C}^1$ vector field
  $\mathbf{F}:\mathbb{R}^m \rightarrow \mathbb{R}^m$, 
  we have
  \begin{equation}
    \label{eq:divThm}
    \oint_{ \mathcal{S} } 
    \mathbf{F}(\mathbf{p}) \cdot
    \mathbf{n}_{\mathcal{S}}(\mathbf{z}_{\mathbf{p}})
    \dif\mathbf{p}
    = \int_{\mathbb{B}^m}
    \nabla\cdot\mathbf{F}(\varphi (\mathbf{x})) J_\varphi
    \dif \mathbf{x}, 
  \end{equation}
  where $J_{\varphi}$ is the determinant
  of the Jacobian matrix $\dif \varphi$ 
  and the outward normal $\mathbf{n}_{\mathcal{S}}$
  of $\mathcal{S}$  
  is given by Definition \ref{def:OutwardNormalVec}. 
\end{theorem}
\begin{remark}
  If $\varphi$ is injective,
  the formula of integration by substitution yields 
  $\int_{\mathbb{B}^m}\nabla\cdot\mathbf{F}(\varphi(\mathbf{x}))
  J_{\varphi}\dif\mathbf{x}
  =\int_{\varphi(\mathbb{B}^m)}\nabla\cdot
  \mathbf{F}(\mathbf{p})\dif\mathbf{p}$, 
  then \eqref{eq:divThm} reduces
  to the classical divergence theorem on $\varphi(\mathbb{B}^m)$. 
  If $\varphi$ is not injective,
  then $\mathcal{S}$ could be self-intersecting
  and one cannot deduce (\ref{eq:divThm})
  from the divergence theorem. 
\end{remark}
\begin{proof}
  Write $\mathbf{p}:=\varphi(\mathbf{x})$
  and we prove (\ref{eq:divThm}) in four steps. 

  First, we show that
  $\nabla\cdot\mathbf{F}(\varphi (\mathbf{x})) J_\varphi$
  is a sum of determinants. 
  Let $F_j$ be the $j$th component of $\mathbf{F}$. 
  For a fixed $j$, the chain rule gives 
  $\sum_{k=1}^{m}
  \frac{\partial \varphi_k}{\partial x_j}\frac{\partial F_j}{\partial p_k}
  = \frac{\partial F_{j}}{\partial x_j}$; 
  these $m$ equations form a linear system, 
  for which Cramer's rule implies
  \begin{equation}
    \label{eq:DetM}
    J_{\varphi}\frac{\partial F_j}{\partial p_j}
    = \det
      \begin{bmatrix}
      \pdfFrac{\varphi_{1}}{x_{1}}&\cdots&
      \pdfFrac{\varphi_{j-1}}{x_{1}}&\pdfFrac{F_{j}}{x_{1}}&
      \pdfFrac{\varphi_{j+1}}{x_{1}}&\cdots&
      \pdfFrac{\varphi_{m}}{x_{1}}\\
      \vdots&\ddots&\vdots&\vdots&\vdots&\ddots&\vdots\\
      \pdfFrac{\varphi_{1}}{x_{m}}&\cdots&
      \pdfFrac{\varphi_{j-1}}{x_{m}}&\pdfFrac{F_{j}}{x_{m}}&
      \pdfFrac{\varphi_{j+1}}{x_{m}}&\cdots&
      \pdfFrac{\varphi_{m}}{x_{m}}
    \end{bmatrix}
    =: \det(M_{j}).
  \end{equation}

  Second, we show that the RHS of (\ref{eq:divThm}) 
  equals the integral of some divergence, 
  \begin{equation}
    \label{eq:divThmStep2}
    I_B := \int_{\mathbb{B}^m}
    \nabla\cdot\mathbf{F}(\varphi (\mathbf{x})) J_\varphi
    \dif \mathbf{x}
    = \int_{\mathbb{B}^m} \nabla\cdot \mathbf{g}(\mathbf{x})\, \dif\mathbf{x},
  \end{equation}
  where the $i$th component of $\mathbf{g}$ is defined as 
  \begin{equation}
    \label{eq:vecG}
    \begin{array}{l}
    g_i := \sum_{j=1}^m (-1)^{i+j}K_{i,j} 
    F_{j}(\varphi(\mathbf{\mathbf{x}}))
    \end{array}
  \end{equation}
  and $K_{i,j}$ denotes the $(i,j)$ cofactor of $\dif\varphi$.
  To prove that (\ref{eq:divThmStep2}) holds, 
  it suffices to show
  $\sum_{i=1}^m \frac{\partial g_i}{\partial x_i}
  = \sum_{j=1}^m \det(M_j)$,
  which follows from
  \begin{displaymath}
    \renewcommand{\arraystretch}{1.3}
    \begin{array}{rl}
      \sum_{j=1}^m\det (M_{j})
      &=\sum_{j=1}^m\sum_{i=1}^{m}(-1)^{i+j}
       \frac{\partial F_j}{\partial x_i}K_{i,j}
      \\
      &=\sum_{j=1}^m\sum_{i=1}^{m}(-1)^{i+j}
        \left[\frac{\partial}{\partial x_{i}}\left(F_{j}K_{i,j}\right)
        -F_j\frac{\partial K_{i,j}}{\partial x_i}\right]
      \\
      &=\sum_{j=1}^m\sum_{i=1}^{m} \frac{\partial}{\partial x_{i}}
        \left[(-1)^{i+j}F_{j}K_{i,j}\right]
        -(-1)^{i+j}F_{j}\frac{\partial K_{i,j}}{\partial x_i}
      \\
      &=\sum_{j=1}^m\sum_{i=1}^{m} \frac{\partial}{\partial x_{i}}
      \left[(-1)^{i+j}F_{j}K_{i,j}\right]
      \\
      &=\sum_{i=1}^m\frac{\partial g_i}{\partial x_i}, 
    \end{array}
  \end{displaymath}
  where the first step follows from the Laplace formula
  applied to (\ref{eq:DetM}),
  the last from (\ref{eq:vecG}), 
  and the penultimate from
\begin{displaymath}
  \begin{array}{rl}
    &\sum_{i=1}^{m}(-1)^{i+j}\frac{\partial K_{i,j}}{\partial x_i}
    = \det\begin{bmatrix}
      \pdfFrac{\varphi_{1}}{x_{1}}&\cdots&
      \pdfFrac{\varphi_{j-1}}{x_{1}}&\pdfFrac{}{x_{1}}&
      \pdfFrac{\varphi_{j+1}}{x_{1}}&\cdots&
      \pdfFrac{\varphi_{m}}{x_{1}}\\
      \vdots&\ddots&\vdots&\vdots&\vdots&\ddots&\vdots\\
      \pdfFrac{\varphi_{1}}{x_{m}}&\cdots&\pdfFrac{\varphi_{j-1}}{x_{m}}&
      \pdfFrac{}{x_{m}}&\pdfFrac{\varphi_{j+1}}{x_{m}}&\cdots&
      \pdfFrac{\varphi_{m}}{x_{m}}
    \end{bmatrix}
    \\
    =&\sum_{(i_1,\ldots,i_m)\in S_{m}}e_{i_1,\ldots,i_m}
    \frac{\partial}{\partial x_{i_{j}}}
    \left(\frac{\partial \varphi_1}{\partial x_{i_{1}}}
      \ldots\pdfFrac{\varphi_{j-1}}{x_{i_{j-1}}}
      \pdfFrac{\varphi_{j+1}}{x_{i_{j+1}}}\ldots\frac{\partial
        \varphi_m}{\partial x_{i_{m}}}\right)\\
    =&\sum_{(i_1,\ldots,i_m)\in S_{m}}
    e_{i_1,\ldots,i_m}\sum_{k=1,k\neq j}^{m}
    \frac{\partial^{2} \varphi_{k}}{\partial x_{i_{k}}\partial x_{i_{j}}}
    \left(\frac{\partial \varphi_1}{\partial x_{i_{1}}}\ldots
      \frac{\partial\hat{\varphi}_j}{\partial \hat{x}_{i_{j}}}
      \ldots\pdfFrac{\hat{\varphi}_{k}}{\hat{x}_{i_{k}}}\ldots
      \frac{\partial\varphi_m}{\partial x_{i_{m}}}\right)\\
    =&0,  
  \end{array}
\end{displaymath}
where $S_{m}$ denotes the symmetry group of order $m$
and $\hat{\varphi}_{j}$ the omitting of $\varphi_{j}$; 
the first equality follows from the Laplace formula
and the last from the symmetry of
$\frac{\partial^{2} \varphi_{k}}{\partial x_{i_{j}}\partial x_{i_{k}}}$
and the anti-symmetry of the Levi-Civita symbol $e_{i_1,i_2,\ldots,i_m}$.
  
Third, we apply the classical divergence theorem to obtain
\begin{displaymath}
  \begin{array}{l}
    I_B = \int_{\mathbb{B}^m} \nabla\cdot \mathbf{g}(\mathbf{x})\, \dif\mathbf{x}
    = \oint_{\partial \mathbb{B}^m} \mathbf{n}_{\partial \mathbb{B}^m}\cdot
    \mathbf{g}(\mathbf{z})\, \dif\mathbf{z}. 
  \end{array}
\end{displaymath}

Last, we define $\tilde{\mathbf{z}}_{r,s}
:= (z_{1},\ldots,z_{r-1},s,z_{r+1},\ldots,z_{m})$,
$B_{r,s}:=\left\{\tilde{\mathbf{z}}_{r,s}
  \ |\ z_{i}\in(0,1)\right\}$,
where $s=0$ or 1, 
and identify $I_B$ with the LHS of (\ref{eq:divThm}):
\begin{displaymath}
  \begin{aligned}
    I_B&=\oint_{\partial \mathbb{B}^m}\mathbf{n}_{\partial \mathbb{B}^m}
    \cdot\mathbf{g}(\mathbf{z})\dif \mathbf{z}
    =\sum_{r=1}^{m}\sum_{s=0}^{1}\int_{B_{r,s}}
    \mathbf{n}_{\partial \mathbb{B}^m}\cdot\mathbf{g}(\mathbf{z})
    \dif \mathbf{z}
    \\
    &=\sum_{r=1}^{m}\sum_{s=0}^{1}(-1)^{s+1}
    \int_{\mathbb{B}^{m-1}}\sum_{j=1}^{m}(-1)^{r+j}F_{j}
    (\varphi(\tilde{\mathbf{z}}_{r,s})) K_{r,j}\dif z_{1}
    \cdots\dif z_{r-1}\dif z_{r+1}\cdots\dif z_{m}
    \\
    &=\sum_{r=1}^{m}\sum_{s=0}^{1}\int_{\varphi(B_{r,s})}
    \mathbf{F}(\mathbf{p})\cdot\mathbf{n}_{\mathcal{S}}
    \dif \mathbf{p}
    =\oint_{\varphi(\partial \mathbb{B}^m)}\mathbf{F}\cdot
    \mathbf{n}_{\mathcal{S}}\dif \mathbf{p}, 
  \end{aligned}
\end{displaymath}
where the second step follows
from the relation between $\mathbb{B}^m$ and $B_{r,s}$; 
the third from (\ref{eq:vecG}),
Definition \ref{def:OutwardNormalVec},
and the fact that the low face $B_{r,0}$
and the high face $B_{r,1}$
respectively contribute to factors $-1$ and $+1$; 
the fourth step from 
Definition \ref{def:OutwardNormalVec}
and the fact that $K_{r,s}$ is the determinant of
the Jacobian matrix for $\varphi(B_{r,s})$; 
and the last step 
from the relation between $\mathbb{B}^m$ and $B_{r,s}$.
\end{proof}


\begin{theorem}
  \label{thm:GeneralizedReynoldsTransportTheorem}
  A moving cycle
  $\mathcal{S}(t):=\varphi(\partial \mathbb{B}^m, t)$
  and a ${\mathcal C}^1$ scalar field $f$
  satisfy
  \begin{equation}
    \label{eq:ReynoldsTransport}
  \frac{\dif}{\dif t} 
  \left( \int_{\mathbb{B}^m} 
    f (\varphi (\mathbf{x},t),t) J_{\varphi}(t)
    \dif  \mathbf{x}\right)
  =
  \int_{\mathbb{B}^m}   \partial_{t}f 
    J_{\varphi}(t)\dif  \mathbf{x}
  +\oint_{\mathcal{S} (t) } f (\mathbf{p},t)
  \partial_{t}\varphi \cdot\mathbf{n}_{\mathcal{S}}\dif\mathbf{p},
  \end{equation}
  where
  $\mathbf{n}_{\mathcal{S}}$ is the outward normal vector 
  of $\mathcal{S}(t)$ as in Definition \ref{def:OutwardNormalVec}. 
\end{theorem}
\begin{proof}
  The time-independence of $\mathbb{B}^m$ 
  and Jacobi's formula (\ref{eq:JacobiEquation}) yield
  \begin{displaymath}
    \renewcommand{\arraystretch}{1.2}
    \begin{array}{rl}
    \frac{\dif }{\dif  t} 
    \left( \int_{\mathbb{B}^m} 
      f 
      J_{\varphi}(t)
    \dif  \mathbf{x} \right)
    &=\int_{\mathbb{B}^m} 
    \left(  \partial_{t}f+ \nabla_{\mathbf{x}} 
    f\cdot \partial_{t}\varphi+ f \nabla_{\mathbf{x}} 
    \cdot (\partial_{t}\varphi)\right) J_{\varphi}(t)
      \dif  \mathbf{x}
      \\
    &=\int_{\mathbb{B}^m}   \partial_{t}f 
    J_{\varphi}(t)\dif  \mathbf{x}+ 
    \int_{\mathbb{B}^m} \nabla_{\mathbf{x}}
    \cdot(f\partial_{t}\varphi) J_{\varphi}(t)
       \dif  \mathbf{x}
    \end{array}
    \end{displaymath}
    and the proof is completed by Theorem \ref{thm:div}. 
\end{proof}

Theorem \ref{thm:GeneralizedReynoldsTransportTheorem}
 is our customized form of Reynolds transport theorem
 for the $m$-cube. 
Although the integral on the LHS of (\ref{eq:ReynoldsTransport})
 has the fixed domain $\mathbb{B}^m$,
 any point $\mathbf{x}\in \mathbb{B}^m$ moves under the action
 of $\varphi$,
 and thus this integral is essentially an integral
 over a moving region. 

\subsection{Donating regions and flux identities}
\label{sec:FluxIdentity}

\begin{definition}
  \label{def:DR}
  The \emph{donating region 
  of a moving hypersurface} $\mathcal{S}(t)$
  in the flow of a $\mathcal{C}^1$ velocity field 
  $\mathbf{u}(\mathbf{x},t)$ 
  over a time interval $(t_{0},t_{0}+k)$ is
  \begin{equation}
    \label{eq:DR}
    \mathcal{D}_{\mathcal{S}}(t_{0},k)
    = \cup_{n\in\mathbb{Z}\setminus\{0\}}
    \mathcal{D}_{\mathcal{S}}^{n}(t_{0},k)
    := \cup_{n\in\mathbb{Z}\setminus\{0\}}
    \{\mathbf{p}\in\mathbb{R}^{m}\;|\;
    \deg(\chi,\mathbb{B}^m,\mathbf{p})=n\}, 
  \end{equation}
  where $\mathcal{D}_{\mathcal{S}}^{n}(t_{0},k)$
  is called the \emph{donating region with index $n$}.
\end{definition}

By a comparison of the donating region $\mathcal{D}_{\mathcal{S}}(t_{0},k)$
 in (\ref{eq:DR})
 to the generating cycle $\mathcal{G}_{\mathcal{D}}(t_{0},k)$
 in (\ref{eq:generatingCycleAsImage})
 and the properties of topological degree presented
 in Section \ref{sec:topological-degree},
 the boundary of $\mathcal{D}_{\mathcal{S}}(t_{0},k)$
 for any index $n\ne 0$
 is a subset of $\mathcal{G}_{\mathcal{D}}(t_{0},k)$. 
Since all donating regions can be determined from
 $\mathcal{G}_{\mathcal{D}}(t_{0},k)$. 
 we consider it as an \emph{explicit construction}
 of donating regions of all indices.

The area formula
 \cite[p. 69]{giaquinta1998cartesian}
 \cite[p. 125]{krantz2008geometric} 
 states that a Lipschitz continuous function
 $\varphi: \Omega\rightarrow \mathbb{R}^{m}$
 and a scalar function $f: \mathbb{R}^m\rightarrow \mathbb{R}$
 satisfy
\begin{equation}
  \label{eq:AreaFormula}
  \int_{\Omega} f (\varphi (\mathbf{x}))
  J_{\mathbf{x}}(\varphi)\dif\mathbf{x}
  = \sum_{n \in \mathbb{Z}\setminus\{0\}}
  n\int_{\mathcal{D}^{n} } f (\mathbf{y})\dif\mathbf{y},
\end{equation}
where $\mathcal{D}^{n}:=\{\mathbf{y}\;|\;
 \deg(\varphi,\Omega,\mathbf{y})=n\}$. 
When $\varphi$ is injective,
the only nonempty $\mathcal{D}^{n}$
is either $\mathcal{D}^{1}$ or $\mathcal{D}^{-1}$; 
then (\ref{eq:AreaFormula}) reduces
to the formula of integration by substitution. 

\begin{theorem}[Flux identities]
  \label{thm:fluxIdentityMain}
  For a $\mathcal{C}^1$ velocity 
  $\mathbf{u}:\mathbb{R}^{m}\times\mathbb{R}\rightarrow\mathbb{R}^m$, 
  a scalar field $f$
  satisfying the conservation law \eqref{eq:ConservationLaw}, 
  and a moving hypersurface
  $\mathcal{S}(t)$, 
  we have 
  \begin{equation}
  \label{eq:fluxIdentityMain}
  \begin{aligned}
    &\int_{t_{0}}^{t_{0}+k}\int_{\mathcal{S}(t)}
    f(\mathbf{x},t)
    \bigl[\mathbf{u} (\mathbf{x},t)- \partial_{t}\mathcal{S}(t)\bigr]
    \cdot \mathbf{n}_{\mathcal{S}}\,\dif \mathbf{x}\dif t
    \\
    =&\sum_{n\in\mathbb{Z}\setminus\{0\}}
    n\int_{\mathcal{D}_{\mathcal{S}}^{n}(t_{0},k)}
    f(\mathbf{x},t_{0})\,\dif\mathbf{x}
    =\oint_{\mathcal{G}_{\mathcal{D}}(t_{0},k)}
    \mathbf{F}(\mathbf{x},t_{0})
    \cdot\mathbf{n}_{\mathcal{G}_{\mathcal{D}}}\,\dif\mathbf{x},
    \end{aligned}
  \end{equation}
  where $\mathbf{n}_{\mathcal{S}}$ is the outward unit normal vector 
  on $\mathcal{S}$ given by Definition \ref{def:OutwardNvecForSurf},
  the donating region $\mathcal{D}_{\mathcal{S}}^{n}(t_{0},k)$
  and the generating cycle $\mathcal{G}_{\mathcal{D}}(t_{0},k)$
  are respectively given by Definitions
  \ref{def:DR} and \ref{def:generatingCycle}, 
  and 
  $\mathbf{F}$ is a vector field satisfying 
  $\nabla \cdot \mathbf{F} (\mathbf{x},t_0)= f (\mathbf{x},t_{0})$, 
  e.g., $\mathbf{F}=\left(\int_{\xi}^{x_{1}}
  f(s,x_{2},\ldots,x_{m},t_{0})\dif s,0,\hdots,0\right)$ 
  where $\xi$ is a fixed real number.
\end{theorem}
\begin{proof}
  Since $t_0$ and $k$ are fixed in this proof, 
  we adopt shorthand notations, 
  \begin{equation}
    \label{eq:shorthandsInFluxIden}
    \forall \tau\in(0,k), \quad
    \mathcal{G}(\tau):=\mathcal{G}_{\mathcal{D}}(t_{0}+\tau,k-\tau),
    \quad
    \mathcal{D}^{n}(\tau)
    :=\mathcal{D}_{\mathcal{S}}^{n}(t_{0}+\tau,k-\tau).
  \end{equation}
  Also, we assume $t_0=0$ since this incurs no loss of generality. 
  
  Consider at time $\tau$
  a particle $\mathbf{p}(\tau)\in \mathcal{G}(\tau)$; 
  its velocity can be expressed as
  $V(\mathbf{p},\tau):=\partial_{\tau}
  \mathcal{G}(\tau)(\mathbf{y}_{\mathbf{p}})$, 
  where $\mathbf{y}_{\mathbf{p}}\in \mathbb{R}^{m-1}$
  is the parameter of $\mathbf{p}(\tau)$
  on $\mathcal{G}(\tau)$.
  Denote by $\mathbf{n}_{\mathcal{G}, \mathbf{p}}(\tau)$
  the normal vector of $\mathcal{G}(\tau)$ at $\mathbf{p}(\tau)$
  according to Definition \ref{def:OutwardNormalVec}
  and define $V_{\mathbf{n}, \mathcal{G}}(\mathbf{p},\tau)
  := V(\mathbf{p},\tau)\cdot \mathbf{n}_{\mathcal{G}, \mathbf{p}}(\tau)$.
  The construction of $\mathcal{G}(\tau)$
  in (\ref{eq:FluxGeneratingHypersurface}) yields 
  \begin{equation}
    \label{eq:velFluxGenCycle}
    V_{\mathbf{n}, \mathcal{G}}(\mathbf{p},\tau)
    = 
    \begin{cases}
      \partial_{t}\mathcal{S}(\mathbf{z}_{\mathbf{p}},\tau)
      \cdot \mathbf{n}_{\mathcal{S}, \mathbf{p}}(\tau)
      & \text{ if } \mathbf{p}(\tau)\in \mathcal{S} (\tau); 
      \\
      \mathbf{u}(\mathbf{p}(\tau),\tau) \cdot
      \mathbf{n}_{\mathcal{G},\mathbf{p}}(\tau)
      & \text{ if } \mathbf{p}(\tau)\in
      \mathcal{G}(\tau)\setminus\mathcal{S} (\tau).
    \end{cases}
  \end{equation}

Consider the rate of change of the integral of $f$
 over donating regions,
\begin{equation}
\label{eq:ProofFluxIdensity}
\renewcommand{\arraystretch}{1.2}
\begin{array}{rl}
  &\frac{\dif}{\dif\tau} \sum_{n}n\int_{\mathcal{D}^{n}(\tau)}
    f(\mathbf{x},\tau)\dif\mathbf{x}
    \\
    =&
    \frac{\dif}{\dif \tau} 
    \left( \int_{\mathbb{B}^m} 
    f (\varphi (\mathbf{x},\tau),\tau) J_{\varphi}(\tau)
    \dif  \mathbf{x}\right)
  \\
  =& \int_{\mathbb{B}^m}\partial_{\tau}f(\mathbf{x},t)
     J_{\varphi}(\tau)\dif  \mathbf{x}
     + \oint_{\mathcal{G}(\tau)}
     f(\mathbf{p},\tau)V_{\mathbf{n}, \mathcal{G}}(\mathbf{p},\tau)
     \dif \mathbf{p}
  \\
  =& -\int_{\mathbb{B}^m} \nabla\cdot(f\mathbf{u}) 
     J_{\varphi}(\tau)\dif  \mathbf{x}
     +\oint_{\mathcal{G}(\tau)}
     f(\mathbf{p},\tau)V_{\mathbf{n}, \mathcal{G}}(\mathbf{p},\tau)
     \dif \mathbf{p}
  \\
  =&-\int_{\mathcal{G}(\tau)}
     f(\mathbf{p},\tau)\mathbf{u}\cdot\mathbf{n}_{\mathcal{G},\mathbf{p}}
     \dif \mathbf{p}
     +\oint_{\mathcal{G}(\tau)}
     f(\mathbf{p},\tau)V_{\mathbf{n}, \mathcal{G}}(\mathbf{p},\tau)
     \dif \mathbf{p}
  \\
  =&\oint_{\mathcal{G}(\tau)}
     f(\mathbf{p},\tau)(V_{\mathbf{n}, \mathcal{G}}(\mathbf{p},\tau)
     -\mathbf{u}\cdot\mathbf{n}_{\mathcal{G},\mathbf{p}})\dif \mathbf{p}
  \\
  =&-\int_{\mathcal{S} (\tau)}
     f(\mathbf{x},\tau)
     \bigl[\mathbf{u} (\mathbf{x},t)- \partial_{t}\mathcal{S}(t)\bigr]
     \cdot\mathbf{n}_{\mathcal{S}}\dif \mathbf{x},
  \end{array}
\end{equation}
where the first step follows
from the area formula (\ref{eq:AreaFormula}),
the second from Theorem 
\ref{thm:GeneralizedReynoldsTransportTheorem}
and (\ref{eq:velFluxGenCycle}), 
the third from the scalar conversation law
\eqref{eq:ConservationLaw}; 
the fourth from Theorem \ref{thm:div}; 
and the last from (\ref{eq:velFluxGenCycle}). 

The first equality in (\ref{eq:fluxIdentityMain})
follows from 
integrating the first and the last lines in \eqref{eq:ProofFluxIdensity} 
over $[0,k]$
while the second equality in (\ref{eq:fluxIdentityMain})
from Theorem \ref{thm:div},
Lemma \ref{lem:generatingCycleAsImage},
and Definition \ref{def:DR}. 
\end{proof}

In our previous work,  we have always
associated the concept of donating regions
with fluxes
and the concept of flux sets
with particle crossings. 
Therefore 
 Theorem \ref{thm:fluxIdentityMain},
 rather than Definition \ref{def:DR}
 and Theorem \ref{thm:fluxingIndexAsMapDegree},
 is the \emph{de facto} proof
 that donating regions and flux sets are index-by-index equivalent,
 i.e., 
\begin{equation}
  \label{eq:equivFluxSetToDonatingRegion}
  \forall n\in\mathbb{Z},\quad
  \mathcal{D}_{\mathcal{S}}^{n}(t_{0},k)
  =\mathcal{F}_{\mathcal{S}}^{n}(t_{0},k).
\end{equation}



\section{Algorithm}
\label{sec:algorithm}

By Theorem \ref{thm:fluxIdentityMain},
 the Eulerian flux of a scalar function $f$
 through a moving surface 
 $\mathcal{S}(u,v,t)\subset\mathbb{R}^{3}$ over $[t_0,t_e]$
 equals a spatial integral 
 over the generating cycle at the initial time $t_0$.
This flux identity gives rise to an LFC algorithm
 that is conceptually very straightforward:
 constructing the generating cycle $\mathcal{G}_{\mathcal{D}}$
 and integrating $f$
 over $\mathcal{G}_{\mathcal{D}}$. 

As a fundamental building block,
 the action of the flow map $\phi_{t_0}^{t_e-t_0}$
 upon a set of isolated points
 is approximated by a $\kappa$th-order ODE solver
 such as an explicit Runge-Kutta method;  
 the algorithmic steps
 are listed in Algorithm \ref{alg:flowMap}. 

\begin{algorithm}
  \caption{\texttt{FlowMap}
  $(\mathbf{u}, \{p_i\}_{i=1}^n, t_0, t_e, \kappa, \Delta t)$}
  \label{alg:flowMap}
  \KwIn{A velocity field $\mathbf{u}(\mathbf{x},t)$, 
  a point set $\{ p_i\}_{i=1}^n,  p_i \in\mathbb{R}^{3}$, 
  the initial time $t_{0}$,
 the ending time $t_{e}$, 
 a $\kappa$th-order time integrator \texttt{ODESolve}, 
 a tentative time step size $\Delta t$.}
  \Pre{$(t_{e}-t_{0})\Delta t>0$, 
  $\mathbf{u}\in {\mathcal C}^{\kappa}(\mathbb{R}^3\times\mathbb{R})$.}
  \KwOut{a finite sequence of points $\{ q_i\}_{i=1}^n$.}
  \Post {$\forall i=1, \hdots, n$, 
  $ \left \| q_i-\phi_{t_{0}}^{t_{e}-t_{0}}(p_{i}) \right \|_{2}
  =\mathcal{O}\left( \left( \Delta t \right)^{\kappa} \right)$.}
  \LinesNumbered
  $m$$\gets$$\lceil\frac{t_{e}-t_{0}}{\Delta t}\rceil$, 
  $\Delta t$$\gets$$\frac{t_{e}-t_{0}}{m}$;
  
  $\{ q_i\}_{i=1}^n \gets \{ p_i\}_{i=1}^n$;
  
  \For{$j=0:m-1$}{
   $\{ q_i\}_{i=1}^n \gets \texttt{ODEsolve}(\mathbf{u}, \{ q_i\}_{i=1}^n,  t_0 + j \Delta t, \Delta t)$;
  }
  \Return {$\{ q_i\}_{i=1}^n$};
\end{algorithm}

\begin{algorithm}
  \caption{\texttt{GeneratingCycle}
  $(\mathbf{u}, \mathcal{S}, t_0, t_e, h, \Delta t, \kappa)$}
  \label{alg:generatingCycle}
  \KwIn{A velocity field $\mathbf{u}(\mathbf{x},t)$, 
	a moving simple surface $\mathcal{S}(u,v,t)$, 
	the initial time $t_{0}$, 
	the ending time $t_{e}$, 
	a spatial length scale $h$, 
	the time increment step $\Delta t$, 
	a $\kappa$th-order time integrator \texttt{ODESolver}.}
      \Pre{$\mathbf{u}\in
        {\mathcal C}^{\kappa}(\mathbb{R}^3\times\mathbb{R})$; 
	$\mathcal{S} \in \mathcal{C}^{\kappa}([0,1]^{2})$, 
        $\forall t\in(t_0,t_e], \mathcal{S}(t)\simeq
        \mathcal{S}(t_0)$, 
	$h>0$, 
	$(t_{e}-t_{0})\Delta t>0$, 
	$\kappa \in \{2, 4, 6\}$.}
      \KwOut{a set of six spline surfaces $\tilde{\mathcal{P}}_{i}$'s 
        whose union approximate the generating cycle 
        $\mathcal{G}_{\mathcal{D}}=\cup_{i=1}^{6}\mathcal{P}_{i}$.}
      \Post{$\forall(u,v)\in[0,1]^{2}$,
        $ \| \tilde{\mathcal{P}}_{i}(u,v)-
          \mathcal{P}_{i}(u,v)  \|_{2}=O(p_{\kappa}(h, \Delta t))$,
    where $p_{\kappa}(h, \Delta t)$ is a $\kappa$th-order polynomial
    in $h$ and $\Delta t$.}

 \LinesNumbered
$M=\lceil\frac{t_e-t_0}{\Delta t}\rceil, N = \lceil\frac{1}{h} \rceil$
\tcp*{Calculate the number of nodes} 

$\{u_i\}_{i=0}^{N}, \{v\}_{j=0}^{N}, 
u_i = \frac{i}{N}, v_j = \frac{j}{N}, i, j = 0, 1,..., N$
 \tcp*{Generate grids} 

$\{p_{i, j}^{1}\}_{i, j = 0}^{N}
\gets \{ \mathcal{S}(u_i, v_j, t_0) \}_{i, j = 0}^{N}$
\tcp*{$\{p_{i, j}^{i}\}_{i, j = 0}^{N}$
  are knots of $\tilde{\mathcal{P}}_{i}$
  in (\ref{eq:partitionGeneratingCycle})}

$\{p_{i, j}^{2}\}_{i, j = 0}^{N}
\gets \texttt{FlowMap}(\mathbf{u}, 
\{ \mathcal{S}(u_i, v_j, t_e) \}_{i, j = 0}^{N}, t_e, t_0, 
\kappa, -\Delta t)$

\For {$j = 0 : M$}{ 
$\{p_{i, j}^{3}\}_{i = 0}^{N}
\gets \texttt{FlowMap}(\mathbf{u}, 
\{ \mathcal{S}(0, v_{i}, t_0+j\Delta t) \}_{i = 0}^{N}, t_0+j \Delta t, 
t_0, \kappa, -\Delta t)$

$\{p_{i, j}^{4}\}_{i = 0}^{N}
\gets \texttt{FlowMap}(\mathbf{u}, 
\{ \mathcal{S}(u_{i}, 1, t_0+j\Delta t) \}_{i = 0}^{N}, 
t_0+j \Delta t, t_0, \kappa, -\Delta t)$

$\{p_{i, j}^{5}\}_{i = 0}^{N} \gets \texttt{FlowMap}
(\mathbf{u}, \{ \mathcal{S}(1,v_{N-i}, t_0+j\Delta t) \}_{i = 0}^{N}, 
t_0+j \Delta t, t_0, \kappa, -\Delta t) $

$\{p_{i, j}^{6}\}_{i = 0}^{N}  \gets \texttt{FlowMap}
(\mathbf{u}, \{ \mathcal{S}(u_{N-i}, 0, t_0+j\Delta t) \}_{i = 0}^{N}, 
t_0+j \Delta t, t_0, \kappa, -\Delta t)$
}

\For{$k = 1 : 6$}{
$\tilde{\mathcal{P}}_{k} \gets$ a $\kappa$ th-order spline fit through $p_{i,j}^{k}$
}
  
\Return $\{\tilde{\mathcal{P}}_{i}\}_{i=1}^{6}$
\end{algorithm}

The generating cycle 
$\mathcal{G}_{\mathcal{D}}=\chi(\partial \mathbb{B}^{3})$
is partitioned into six surfaces, 
\begin{equation}
  \label{eq:partitionGeneratingCycle}
\begin{array}{l}
\mathcal{P}_{1}=\chi((0,1)\times(0,1)\times\{0\}),\quad 
  \mathcal{P}_{2}=\chi((0,1)\times(0,1)\times\{1\}),
  \\
\mathcal{P}_{3}=\chi(\{0\}\times(0,1)\times(0,1)), \quad 
  \mathcal{P}_{4}=\chi((0,1)\times\{1\}\times(0,1)),
  \\
\mathcal{P}_{5}=\chi(\{1\}\times(0,1)\times(0,1)),\quad  
\mathcal{P}_{6}=\chi((0,1)\times\{0\}\times(0,1)), 
\end{array}
\end{equation}
each of which is approximated by a bivariate tensor-product spline
with the $\kappa$th-order accuracy. 
Then we assemble these splines
into a discrete approximation of $\mathcal{G}_{\mathcal{D}}$;
see Algorithm \ref{alg:generatingCycle} for more details.

To avoid the discontinuity of tangent spaces
 at the common boundaries of the six splines, 
 we obtain the Lagrangian flux
 by summing up the integrals
 of the scalar function over the six spline surfaces.
Each multi-dimensional integral is calculated
 by recursively applying standard one-dimensional
 Gauss-Legendre rules.
The following lemma details the algorithmic steps
 and guarantees the accuracy.
 
 \begin{lemma}
   \label{lem:DRCubature3d}
   Consider the integral of a trivariate polynomial $f(x,y,z)$
   over a tensor-product spline
   $\mathcal{S}(u,v)=\left\{ x(u,v),y(u,v),z(u,v) \right\}$
   oriented by Definition \ref{def:OutwardNvecForSurf},
   \begin{equation}
     \label{eq:integralToApprox}
     \begin{array}{l}
       I_{\mathcal{S}}(f) :=
     \int_{\mathcal{S}} F(x,y,z)\,\dif y\wedge \dif z, 
     \end{array}
   \end{equation}
   where $F(x,y,z)=\int_{\xi}^{x}f(s,y,z)\,\dif s$
   and $\xi$ is a fixed real number. 
   If ${\mathcal S}$ is of the $\kappa$th order
   and the total degree of $f$ is no greater than $q$, 
   then the choices
   \begin{equation}
     \label{eq:numQuadPoints}
     n\ge \left\lceil\frac{q+1}{2}\right\rceil\quad \text{ and }\quad
     m,h\ge \left\lceil\frac{1}{2}(q+3)(\kappa-1)\right\rceil
   \end{equation}
   yield an exact calculation of $I_{\mathcal{S}}(f)$,
   i.e., 
 \begin{equation}
      \label{eq:quadrature_rule}
      I_{\mathcal{S}}(f)
      = I_{q}(\mathcal{S},f)
      := \sum_{l_{1},l_{2}=1}^{\iota}\sum_{i=1}^{n}
      \sum_{j=1}^{m} \sum_{k=1}^{h}w_{l_{1}l_{2}ijk} 
      f(x_{l_{1}l_{2}ijk}, y_{l_{1}l_{2}jk},z_{l_{1}l_{2}jk}),
  \end{equation}
  where $\left\lceil \cdot \right\rceil$ denotes the ceiling function, 
  $\{\lambda_{i}^p\}_{i=1}^p$ and $\{\omega_{i}^p\}_{i=1}^p$ 
  respectively the nodes and weights
  of the $p$-points Gauss-Legendre rule over $[-1,1]$, 
  and 
  \begin{displaymath}
    \begin{aligned}
      &u_{l_{1}j}:=\frac{u_{l_{1}}-u_{l_{1}-1}}{2}\lambda_{j}^{m} 
      +\frac{u_{l_{1}}+u_{l_{1}-1}}{2},\quad
      v_{l_{2}k}:=\frac{v_{l_{2}}-v_{l_{2}-1}}{2}\lambda_{k}^{h} 
      +\frac{v_{l_{2}}+v_{l_{2}-1}}{2},
      \\
      &x_{l_{1}l_{2}ijk}:= \frac{x(u_{l_{1}j},v_{l_{2}k})-\xi}{2}
      \lambda_{i}^{n}+ \frac{x(u_{l_{1}j},v_{l_{2}k})+\xi}{2},
      \\
      &y_{l_{1}l_{2}jk}:=y(u_{l_{1}j},v_{l_{2}k}), 
      \quad z_{l_{1}l_{2}jk}:=z(u_{l_{1}j},v_{l_{2}k}),
      \\
      &w_{l_{1}l_{2}ijk}:=-\omega_i^n \omega_j^m\omega_k^{h}
      \frac{u_{l_{1}}-u_{l_{1}-1}}{2}\cdot\frac{v_{l_{2}}-v_{l_{2}-1}}{2}
      \cdot\frac{x(u_{l_{1}j},v_{l_{2}k})-\xi}{2}
      \left|\frac{\partial(y,z)}{\partial(u,v)}(u_{l_{1}j},v_{l_{2}k})\right|.
    \end{aligned}
  \end{displaymath}
\end{lemma}
\begin{proof}
  Denote by $\left\{ u_{i} \right\}_{i=0}^{\iota },
  \left\{ v_{i} \right\}_{i=0}^{\iota }$ knots of the spline  
  and express a single piece of polynomial surface as
  $\mathcal{S}_{l_{1},l_{2}}
  :=\mathcal{S}|_{[u_{l_{1}-1},u_{l_{1}}]\times[v_{l_{2}-1},v_{l_{2}}]}$. 
  Converting the surface integral 
  to an integral over the cube $[-1,1]^{3}$, we have 
  \begin{displaymath}
    \begin{array}{l}
      -\sum_{l_{1},l_{2}=1}^{\iota}\int_{\mathcal{S}_{l_{1},l_{2}}}
      {F(x,y,z)\,\dif y \wedge \dif z}
      \\
      =\sum_{l_{1},l_{2}=1}^{\iota}\int_{u_{l_{1}-1}}^{u_{l_{1}}} 
      \int_{v_{l_{2}-1}}^{v_{l_{2}}}F(x,y,z)\left|\frac{\partial(y,z)}
      {\partial(u,v)} \right|\,\dif u \dif v
      \\
      =\sum_{l_{1},l_{2}=1}^{\iota}\int_{-1}^{1}
      \int_{-1}^{1}F(x,y,z)
      \left|\frac{\partial(y,z)}{\partial(\tilde{u},\tilde{v})}\right|
      \,\dif \tilde{u}_{l_{1}} \dif \tilde{v}_{l_{2}}
      \\
      =\sum_{l_{1},l_{2}=1}^{\iota}\int_{-1}^{1}
      \int_{-1}^{1}\int_{\xi}^{x}f(s,y,z)\,\dif s 
      \left|\frac{\partial(y,z)}{\partial(\tilde{u},\tilde{v})}\right|
      \,\dif \tilde{u}_{l_{1}} \dif \tilde{v}_{l_{2}}
      \\
      =\sum_{l_{1},l_{2}=1}^{\iota}\int_{[-1,1]^{3}}^{}
      f(s(\lambda),y(u,v),z(u,v))
      \left|\frac{\partial(y,z)}{\partial(\tilde{u},\tilde{v})}\right|
      \frac{x(u,v)-\xi}{2}\,\dif\lambda
      \dif\tilde{u}_{l_{1}}\dif \tilde{v}_{l_{2}}
      \\
      =\sum_{l_{1},l_{2}=1}^{\iota}\int_{[-1,1]^{3}}f(s(\lambda),y(u,v),z(u,v))
      \frac{x(u,v)-\xi}{2}\frac{u_{l_{1}}-u_{l_{1}-1}}{2}
      \frac{v_{l_{2}}-v_{l_{2}-1}}{2}
      \left|\frac{\partial(y,z)}{\partial(u,v)}\right|
      \,\dif\lambda \dif\tilde{u}_{l_{1}} \dif \tilde{v}_{l_{2}}, 
    \end{array}
  \end{displaymath}
where the first step follows from changes of variables 
and Definition \ref{def:OutwardNvecForSurf}, 
the second from changes of variables,
\begin{displaymath}
\begin{array}{l}
  u(\tilde{u}_{l_{1}})=\frac{u_{l_{1}}-u_{l_{1}-1}}{2}\tilde{u}_{l_{1}}
+\frac{u_{l_{1}}+u_{l_{1}-1}}{2}, \quad
v(\tilde{v}_{l_{2}})=\frac{v_{l_{2}}-v_{l_{2}-1}}{2}\tilde{v}_{l_{2}}
+\frac{v_{l_{2}}+v_{l_{2}-1}}{2},
\end{array}
\end{displaymath}
the fourth from the transformation 
$s(\lambda)=\frac{x-\xi}{2}\lambda+ \frac{x+\xi}{2}$,
and the fifth from 
\begin{displaymath}
  \begin{array}{l}
  \left|\frac{\partial(\tilde{u},\tilde{v})}{\partial(u,v)}\right|
  =\frac{u_{l_{1}}-u_{l_{1}-1}}{2}\frac{v_{l_{2}}-v_{l_{2}-1}}{2}.
  \end{array}
\end{displaymath}


We still need to show
that $I_{q}(\mathcal{S},f) = I_{\mathcal{S}}(f)$ holds. 
Because the degree of exactness of a $p$-points Gauss quadrature formula 
 is $2p-1$,
 we only need to verify that the choices in (\ref{eq:numQuadPoints})
 are sufficiently large,
 which indeed holds because
 degrees of the integrand in the last step of the above equation array 
 in term of $\lambda$, $\tilde{u}_{l_{1}}$, and $\tilde{v}_{l_{2}}$
 are $q$, $(q+3)(\kappa-1)-1$, and $(q+3)(\kappa-1)-1$. respectively. 
\end{proof}

We sum up our LFC algorithm in Algorithm \ref{alg:MLFC}
and formally prove its accuracy in Theorem \ref{thm:accuracyOfLFC}. 

\begin{algorithm}[t]
  \caption{\texttt{LFC3D}
  $(f, \mathbf{u}, \mathcal{S},  t_0, t_e, \texttt{nNodeS}, \texttt{nNodeT}, \kappa)$}
  \label{alg:MLFC}
  \KwIn{A scalar function $f$, 
	a velocity field $\mathbf{u}(\mathbf{x}, t)$, 
	a parameterized surface $\mathcal{S}(u,v,t)$, 
	the initial time $t_{0}$, 
	the ending time $t_e$, 
	the number of spatial subintervals \texttt{nNodeS}, 
	the number of temporal subintervals \texttt{nNodeT}, 
	a $\kappa$th-order ODESolver.}
      \Pre{
        $t_{e}>t_{0}$, $\kappa \in \{2, 4, 6\}$, 
        $\texttt{nNodeS}\in\mathbb{N}^+$,
        $\texttt{nNodeT}\in\mathbb{N}^+$;
        $\mathbf{u}\in{\mathcal C}^{\kappa}$, 
        $f(\mathbf{x},t_{0})\in \mathcal{C}^{\kappa+1}$, 
        $f$ and $\mathbf{u}$ satisfy \eqref{eq:ConservationLaw}; 
	$u\in[0,1]$, 
	$v\in[0,1]$,
        $\mathcal{S}(t_0)\simeq \mathcal{S}(t_e)$,
        $\mathcal{S}(t_0)\in \mathcal{C}^{\kappa}((0,1)^{2})$,
        $\mathcal{S}(t_e)
        \in \mathcal{C}^{\kappa}((0,1)^{2})$;
      }
      \KwOut{$I_{\kappa}$ as an estimate of the Eulerian flux.
        \begin{displaymath}
          \begin{array}{l}
            I_{E} = \int_{t_{0}}^{t_{e}}\int_{\mathcal{S}}
            f(\mathbf{x},t)(\mathbf{u}(\mathbf{x},t)
            -\partial_{t}\mathcal{S}(t)) \cdot \mathbf{n}_{\mathcal{S}}
            \,\dif \mathbf{x} \,\dif t.            
          \end{array}
        \end{displaymath}}
      \Post {
        $I_{\kappa}$ is a $\kappa$th-order approximation to $I_E$. 
        }

      \LinesNumbered
      $h\gets\frac{1}{\texttt{nNodeS}}$, 
      $\Delta t\gets\frac{t_{e}-t_{0}}{\texttt{nNodeT}}$;

      $\{\tilde{\mathcal{P}}_i\}_{i=1}^6 \gets 
      \texttt{GeneratingCycle}(\mathbf{u}, 
      \mathcal{S}, t_0, t_e, h, \Delta t, \kappa)$

      \For{$i=1:6$}{
        $I_{i}\gets$
        compute $I_{q}(\mathcal{\tilde{P}}_{i},f(\mathbf{x},t_{0}))$
        by (\ref{eq:quadrature_rule}) with $q=\kappa$;}

\Return {$I_{1}-I_{2}+I_{3}+I_{4}+I_{5}+I_{6}$};
\end{algorithm}

\begin{theorem}
  \label{thm:accuracyOfLFC}
  The output of Algorithm \ref{alg:MLFC}.
   is a $\kappa$th-order accurate approximation to the Eulerian flux $I_E$,
   i.e., the LHS in (\ref{eq:fluxIdentityMain}).
  More precisely,
   for the asympotic choice of $\Delta t= O(h)$,
   i.e., $\texttt{nNodeT}=O(\texttt{nNodeS})$
   in calling Algorithm \ref{alg:MLFC}, 
   we have
  \begin{equation}
    \label{eq:accuracyOfLFC}
    \left|I_{\kappa}-I_{E}\right| =\mathcal{O}(h^{\kappa}). 
  \end{equation}
\end{theorem}
\begin{proof}
  First, we show that the integral of $f$ is calculated
  to the $\kappa$th-order accuracy
  over each of the six surfaces
  that constitute the generating cycle $\mathcal{G}_{\mathcal{D}}$, 
  i.e., 
\begin{equation}
  \label{eq:approx_surface_integral}
  \forall \mathcal{P} \in \mathcal{C}^{\kappa}((0,1)^{2}), \quad
  \left|I_{\mathcal{P}}(f)- I_{\kappa}(\tilde{\mathcal{P}},f)\right|
  =\mathcal{O}(h^{\kappa}),
\end{equation}
where $\tilde{\mathcal{P}}$ is a $\kappa$th order 
tensor-product spline approximation of $\mathcal{P}$
and $h$ is the maximal distance between two adjacent knots of the spline.
Since $f(\mathbf{x},t_{0})\in \mathcal{C}^{\kappa+1}(\mathbb{R}^{3})$,
the function $F(x,y,z)=\int_{\xi}^{x}f(s,y,z)\dif s$ satisfies 
\begin{equation}
  \label{eq:approx_F}
  \exists C>0 \text{ s.t. }
  \forall u,v\in (0,1),\ 
  |F(\mathcal{P}(u,v))-F(\tilde{\mathcal{P}}(u,v))|\le 
  C\|\mathcal{P}(u,v)-\tilde{\mathcal{P}}(u,v)\|_{2}
\end{equation}
where $\|\mathbf{x}\|_{2}$ is the 2-norm of $\mathbf{x}$. 
Write $\mathbf{x}_{uv}=\mathcal{P}(u,v)$, 
$\tilde{\mathbf{x}}_{uv}=\tilde{\mathcal{P}}(u,v)$, 
and we have 
\begin{equation}
  \label{eq:spline_approximate_affect}
  \renewcommand{\arraystretch}{1.3}
  \begin{array}{rl}
    &|I_{\mathcal{P}}(f)-I_{\tilde{\mathcal{P}}}(f)|
    \le \int_{0}^{1}\int_{0}^{1}
    |F(\mathbf{x}_{uv})J_{\mathcal{P}}
    -F(\tilde{\mathbf{x}}_{uv})J_{\tilde{\mathcal{P}}}|\,\dif u\dif v
    \\
    \le&\int_{0}^{1}\int_{0}^{1}
         |F(\mathbf{x}_{uv})-F(\tilde{\mathbf{x}}_{uv})|\cdot
         |J_{\tilde{\mathcal{P}}}|\,\dif u\dif v
         +\int_{0}^{1}\int_{0}^{1}
         |F(\mathbf{x}_{uv})(J_{\mathcal{P}}-J_{\tilde{\mathcal{P}}})|
         \,\dif u\dif v
    \\
    \le& C\max\limits_{(u,v)\in[0,1]^{2}}\|\mathbf{x}_{uv}
         -\tilde{\mathbf{x}}_{uv}\|_{2}
         \int_{0}^{1}\int_{0}^{1}|J_{\tilde{\mathcal{P}}}|\,\dif u\dif v
         +K\int_{0}^{1}\int_{0}^{1}
         |J_{\mathcal{P}}-J_{\tilde{\mathcal{P}}}|\,\dif u\dif v
    \\
    \le & K_{1}\mathcal{O}(h^{\kappa})+K_{2}\mathcal{O}(h^{2\kappa-2})
          =\mathcal{O}(h^{\kappa}),
  \end{array}
\end{equation}
where the first step follows from a change of variables, 
the second from the triangular inequality,
the third from \eqref{eq:approx_F} and the boundedness of $F$ 
on $\mathcal{P}\cup\tilde{\mathcal{P}}$, 
the fourth from the $\kappa$th-order accuracy
of each one-dimensional spline,
and the last from the condition $\kappa\ge 2$.
In the last two lines, $C,K,K_{1},K_{2}$ are constants. 

Let $p_{ij}$ be the $\kappa$th-degree
 interpolation polynomial of $f$. 
Then 
\begin{equation}
  \label{eq:Cauchy_Remainder}
  \forall u\in(u_{i-1},u_{i}),v\in(v_{j-1},v_{j}),\quad 
  |f(\tilde{\mathcal{P}}(u,v))-p_{ij}(\tilde{\mathcal{P}}(u,v))|
  =O(h^{\kappa+1}).
\end{equation}
Meanwhile, the interpolation conditions dictate 
 $p_{ij}(\tilde{\mathcal{P}}(u_i,v_j))
=f(\tilde{\mathcal{P}}(u_i,v_j))$
 at each interpolation site $(u_i,v_j)$.
Therefore, we have
\begin{equation}
  \label{eq:exact_interpolation}
  I_{\kappa}(\tilde{P},p_{ij})
  =I_{\kappa}(\tilde{P},f). 
\end{equation}


Write $m_{l_1,l_{2}}:=\min\{x|(x,y,z)\in\mathcal{P}_{l_{1},l_{2}}\}$, 
$M_{l_{1},l_{2}}:=\max\{x|(x,y,z)\in\mathcal{P}_{l_{1},l_{2}}\}$, 
then
\begin{displaymath}
  \renewcommand{\arraystretch}{1.3}
  \begin{array}{rl}
    &|I_{\tilde{\mathcal{P}}}(f)-I_{\kappa}(\tilde{\mathcal{P}},f)|
      = |I_{\tilde{\mathcal{P}}}(f)-I_{\kappa}(\tilde{\mathcal{P}},p_{ij})|
      = \int_{\tilde{\mathcal{P}}}|f-p_{ij}|
   \\ \le& \sum_{l_{1},l_{2}=1}^{\iota}
    \int_{\tilde{\mathcal{P}}_{l_{1},l_{2}}}
    \int_{m_{l_{1},l_{2}}}^{M_{l_{1},l_{2}}}
    |f(s,y,z)-p_{ij}(s,y,z)|\dif s
    \dif y\wedge\dif z
    \\
    \le&\sum_{l_{1},l_{2}=1}^{\iota}
    \int_{\tilde{\mathcal{P}}_{l_{1},l_{2}}}
    \mathcal{O}(h^{\kappa+2})\dif y\wedge\dif z
    \le \mathcal{O}(h^{\kappa}),
  \end{array}
\end{displaymath}
where the first step follows from \eqref{eq:exact_interpolation},
the second from the first equality in (\ref{eq:quadrature_rule}),
the third from breaking the integral into pieces, 
the fourth from \eqref{eq:Cauchy_Remainder} and the definition of $h$, 
and the last from $\iota=\mathcal{O}(\frac{1}{h})$.
Then
\eqref{eq:approx_surface_integral} follows from 
\eqref{eq:spline_approximate_affect}
and the triangular inequality. 

It remains to justify signs of individual integrals
 in the last line of Algorithm \ref{alg:MLFC}. 
The flux identity \eqref{eq:fluxIdentityMain} implies 
\begin{equation}
  \label{eq:ApplyFluxIdentity}
  I_{E}=\oint_{\mathcal{G}_{\mathcal{D}}}F(x,y,z)\dif y\wedge\dif z,
\end{equation}
where $\mathcal{G}_{\mathcal{D}}$ is oriented
by Definition \ref{def:OutwardNormalVec} 
whereas the six spline surfaces are oriented
by Definition \ref{def:OutwardNvecForSurf}.
According to Algorithm \ref{alg:generatingCycle}, 
the parametrizations of
$\{\tilde{\mathcal{P}}_{i}\}_{i=1}^{6}$
are the same as those in (\ref{eq:partitionGeneratingCycle})
and thus we have
\begin{displaymath}
  \renewcommand{\arraystretch}{1.3}
  \begin{array}{l}
    \pdfFrac{\tilde{\mathcal{P}}_{1}}{u}
    \wedge\pdfFrac{\tilde{\mathcal{P}}_{1}}{v}
    \wedge \mathbf{n}<0,\quad
    \pdfFrac{\tilde{\mathcal{P}}_{2}}{u}
    \wedge\pdfFrac{\tilde{\mathcal{P}}_{2}}{v}
      \wedge \mathbf{n}<0,\quad
    \pdfFrac{\tilde{\mathcal{P}}_{3}}{v}
    \wedge\pdfFrac{\tilde{\mathcal{P}}_{3}}{t}
    \wedge \mathbf{n}<0,
    \\
    \pdfFrac{\tilde{\mathcal{P}}_{4}}{u}
    \wedge\pdfFrac{\tilde{\mathcal{P}}_{4}}{t}
    \wedge \mathbf{n}<0,\quad
    \pdfFrac{\tilde{\mathcal{P}}_{5}}{u}
    \wedge\pdfFrac{\tilde{\mathcal{P}}_{5}}{t}
    \wedge \mathbf{n}>0,\quad
    \pdfFrac{\tilde{\mathcal{P}}_{6}}{v}
    \wedge\pdfFrac{\tilde{\mathcal{P}}_{6}}{t}
    \wedge \mathbf{n}>0. 
  \end{array}
\end{displaymath}

Hence $\tilde{\mathcal{P}}_{2}$
is the only spline surface whose orientation
is inconsistent to that of $\mathcal{G}_{\mathcal{D}}$.
Therefore, the sign of $I_2$ is $-1$
while that of any other $I_j$ is $+1$. 
\end{proof}

We conclude this section
with a number of observations.
First, in Lemma \ref{lem:DRCubature3d},
one can replace $F(x,y,z)$ 
by any anti-derivative of $f(x,y,z)$ with respect to $x$. 
For example, $\xi$ could be a polynomial of $y$ and $z$; 
but then $m,h$ must be rederived
if the degree of $\xi(y,z)$ is greater than $\kappa-1$.  
Second, the proof of Lemma \ref{lem:DRCubature3d}
  shows that the triple integral over the cuboid $\mathbb{B}^3$
  is theoretically equivalent
  to the surface integral $I_q({\mathcal S},f)$.
  However, the latter is advantageous in numerical calculation
  since it has a lower complexity.
Third, Sudhakar and colleagues
  \cite{sudhakar2013quadrature,sudhakar2014accurate}
  proposed quadrature formulas
  for calculating integrals over a domain bounded
  by linear polygonal faces. 
  The formula we proposed in Lemma \ref{lem:DRCubature3d} 
  can be considered as a generalization
  since it works for domains bounded
  by high-order spline surfaces.
Last, when the number of polynomial surfaces in a spline is large, 
  the number of quadrature nodes in (\ref{eq:quadrature_rule}) 
  may become much larger than the optimum value 
  $\binom{q+3}{3} =\frac{q(q+1)(q+2)}{6}$.
  This could lead to an efficiency deterioration, 
  especially when evaluating the scalar function is expensive 
  and/or (\ref{eq:quadrature_rule}) is repeatedly applied.
  Fortunately, this deterioration can be very much alleviated
  by the compression technique
  based on multivariate discrete measures
  and developed by Sommariva and colleagues 
  \cite{sommariva2015compression,sudhakar2017use}.



\section{Tests}
\label{sec:tests}

In this section,
 we perform a number of numerical tests
 to confirm our analysis
 and to demonstrate the accuracy and efficiency
 of our LFC algorithm.
These tests cover
 static and deforming surfaces,
 incompressible and compressible flows,
 simple and self-intersecting generating cycles.

The relative error of the LFC algorithm is defined as
\begin{equation}
  \label{eq:DefRelativeError}
  E_{\kappa}(h,\Delta t):=
  \left| \frac{I_{\kappa}(h,\Delta t)-I_{E}}{I_{E}}\right|,
\end{equation}
where $I_{E}$ is the exact value of the flux 
and $I_{\kappa}(h,\Delta t)$ is the computational result 
 of the $\kappa$th order LFC algorithm
 with the spatial grid length scale $h$ 
 and the time step size $\Delta t$ as input parameters. 
Based on (\ref{eq:DefRelativeError}), 
 the numerical convergence rate is defined as 
\begin{equation}
  \label{eq:DefConvergenceRate}
  \mathcal{O}_{\kappa}(h,\Delta t)
  =\log_r\frac{E_{\kappa}(rh,r\Delta t)}{E_{\kappa}(h,\Delta t)}, 
\end{equation}
where the refinement factor is set to $r = 2$ in this work.
For the constructed generating cycle
 to be respectively second-, fourth-, and sixth-order accurate, 
 we use linear, cubic, and quintic splines
 for three-dimensional interpolation 
 and adopt the modified Euler method \cite{lambert1991numerical}, 
 the classic fourth-order Runge-Kutta method, 
 and the sixth-order Verner method \cite{verner1978explicit} 
 for time integration.

In all tests, we choose $\texttt{nNodeS}=\texttt{nNodeT}$, 
 i.e. $\Delta t=(t_{e}-t_{0})h$;
 see Algorithm \ref{alg:MLFC}.
We also use the shorthand notation
 $E_{\kappa}(h):=E_{\kappa}(h,(t_{e}-t_{0})h)$.

\subsection{An unsteady incompressible flow}
\label{sec:unst-incompr-flow}

The velocity field for this test
 is the following unsteady incompressible flow 
 proposed by LeVeque \cite{leveque1996high}: 
\begin{equation}
  \label{eq:divFreeVelLeveque}
  \mathbf{u} (\mathbf{x},t)=
  \cos\left(\frac{\pi t}{T}\right)
  \begin{bmatrix}
    2\sin^2(\pi x)\sin(2\pi y)\sin(2\pi z)\\ 
    -\sin(2\pi x)\sin^2(\pi y)\sin(2\pi z)\\ 
    -\sin(2\pi x)\sin(2\pi y)\sin^2(\pi z)
\end{bmatrix}.
\end{equation}

The scalar conservation law \eqref{eq:ConservationLaw} 
 holds for the above velocity and the scalar 
\begin{equation}
  \label{eq:scalarLeveque}
  f(x,y,z,t) = \sin(\pi x)\sin(\pi y)\sin(\pi z).
\end{equation}

Due to the temporal factor $\cos\left(\frac{\pi t}{T}\right)$ 
 in \eqref{eq:divFreeVelLeveque}, 
 the system returns to its initial state at $t=2T$. 

\subsubsection{A static surface}

 For the static surface
 \begin{equation}
  \label{eq:Fixedsurf1}
  \mathcal{S}(u,v,t)=\left(\frac{u}{2},
  \frac{v}{2},\frac{1}{4}\right),\quad u,v\in(0,1), 
 \end{equation} 
 we choose $T=3$, $[t_0,t_e]=[0, \frac{3}{2}]$
 and plot the generating cycle in Figure \ref{fig:divFreeVelLeveque}, 
 where complex geometric features of the donating region
 is demonstrated.
We also report
 corresponding errors and convergence rates of our LFC algorithm
 in Table \ref{tab:divFreeVelLeveque},
 where the second-, fourth-, and sixth-order
 convergence are clearly obtained.

\begin{figure}
  \centering
  \captionsetup{font={small}}  
  \captionof{table}{Errors and convergence rates of the LFC 
    Algorithm \ref{alg:MLFC} for calculating
    the flux of the scalar function 
    \eqref{eq:scalarLeveque}
    through the fixed surface \eqref{eq:Fixedsurf1} 
    in the flow of \eqref{eq:divFreeVelLeveque}
    during the time interval $[0,\frac{3}{2}]$. 
  }
  \begin{tabular}{c|cccccccc}
    \hline \hline
    $\kappa$ 
    & $\mathit{E}_{\kappa}\left(\frac{1}{32}\right)$
    & $\mathcal{O}_{\kappa}$
    & $\mathit{E}_{\kappa}\left(\frac{1}{64}\right)$
    &$\mathcal{O}_{\kappa}$
    & $\mathit{E}_{\kappa}\left(\frac{1}{128}\right)$
    &$\mathcal{O}_{\kappa}$
    & $\mathit{E}_{\kappa}\left(\frac{1}{256}\right)$  \\
    \hline
    2 & 9.15e-03 &   1.69 & 2.83e-03 &   1.87 
    & 7.75e-04 &   1.95 & 2.01e-04 \\  
    4 & 9.63e-04 &   5.13 & 2.75e-05 &   3.98 
    & 1.75e-06 &   4.20 & 9.48e-08 \\ 
    6 & 7.16e-04 &   6.42 & 8.38e-06 &   8.07 
    & 3.12e-08 &   5.86 & 5.38e-10 \\  
    \hline
    \hline
  \end{tabular}
  \label{tab:divFreeVelLeveque}
  \includegraphics[width=0.6\linewidth]{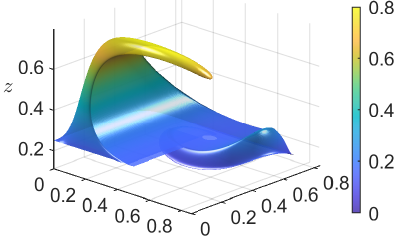}
  \captionof{figure}{The generating cycle
    constructed with $\kappa=6$ and $h=\frac{1}{256}$
    for the fixed surface \eqref{eq:Fixedsurf1}
    over the time interval $[t_0, t_e]=[0,\frac{3}{2}]$
    in the flow of \eqref{eq:divFreeVelLeveque} where $T=3$. 
    The surface color indicates the altitude. 
  }		
  \label{fig:divFreeVelLeveque}
\end{figure}

\subsubsection{A moving surface}

Aside from static surfaces, 
 we also test our LFC algorithm for a moving surface, 
 \begin{equation}
   \mathcal{S}(u, v, t)=\left(  
     u\sin\frac{\pi t}{2},\ 
     v\sin\frac{\pi t}{2},\ 
     \frac{t^2}{72} \left(9u^2+4v^2\right)
   \right),\quad u,v\in(-1, 1).
   \label{eq:MS1}
 \end{equation}

To satisfy the range of $u,v\in(0,1)^{2}$
 assumed in Algorithm \ref{alg:MLFC}, 
 we convert \eqref{eq:MS1} to the following equivalent form
\begin{displaymath}
  \mathcal{S}(u, v, t)=\left(  
    (2u-1)\sin\frac{\pi t}{2},\ 
    (2v-1)\sin\frac{\pi t}{2},\ 
    \frac{t^2}{72}[9(2u-1)^2+4(2v-1)^2]
  \right).
\end{displaymath}

\begin{figure}
  \centering
  \captionsetup{font={small}}
  \captionof{table}{Errors and convergence rates of the LFC 
    Algorithm \ref{alg:MLFC} for calculating
    the flux of scalars
    \eqref{eq:divFreeVelLeveque} and $f=1$
    through the moving surface \eqref{eq:MS1}
    in the flow of \eqref{eq:divFreeVelLeveque}
    during the time interval $[0,1]$; 
    see the last plot in the second row
    of Figure \ref{fig:divFreeVelLevequeMoving}
    for the corresponding generating cycle.
  }    
  \begin{tabular}{c|ccccccc}
    \hline \hline
    $\kappa$    & $E_{\kappa}\left(\frac{1}{32}\right)$
    & ${\mathcal{O}}_{\kappa}$ &$E_{\kappa}\left(\frac{1}{64}\right)$
    & ${\mathcal{O}}_{\kappa}$ &$E_{\kappa}\left(\frac{1}{128}\right)$
    & ${\mathcal{O}}_{\kappa}$ &$E_{\kappa}\left(\frac{1}{256}\right)$
    \\
    \hline
    \multicolumn{8}{c}{the scalar function in \eqref{eq:scalarLeveque}}\\
    \hline\hline
    2 & 5.76     & 2.11 & 1.34     & 2.04 & 3.25e-01 & 2.02 & 8.03e-02\\
    4 & 2.04e-02 & 0.19 & 1.79e-02 & 4.29 & 9.15e-04 & 4.08 & 5.41e-05\\
    6 & 2.57e-01 & 3.91 & 1.71e-02 & 8.69 & 4.15e-05 & 8.57 & 1.09e-07\\
    \hline\hline
    \multicolumn{8}{c}{the constant scalar $f=1$}\\
    \hline\hline
    2 & 5.10e-02 & 2.03 & 1.25e-02 & 2.02 & 3.09e-03 & 2.01 & 7.67e-04\\
    4 & 1.08e-02 & 7.84 & 4.72e-05 & 3.30 & 4.78e-06 & 4.43 & 2.22e-07\\
    6 & 1.49e-02 & 5.71 & 2.83e-04 & 7.94 & 1.16e-06 & 9.02 & 2.23e-09\\
    \hline\hline
  \end{tabular}
  \label{ta:taylar2}

  \vspace{3mm}
  
  \includegraphics[width=1\linewidth]{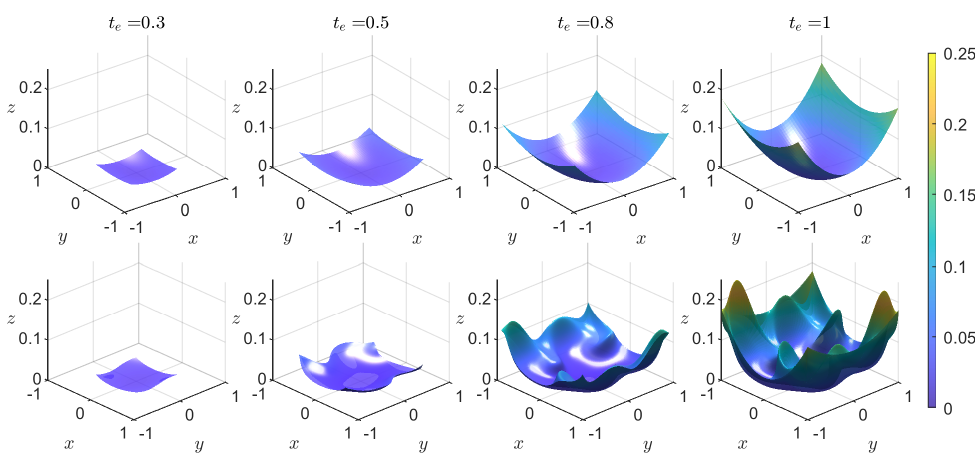} 
  \captionsetup{font={small}}
  \captionof{figure}{The moving surface \eqref{eq:MS1} (the first row)
    and the corresponding generating cycles (the second row)
    constructed by Algorithm \ref{alg:generatingCycle}
    with $\kappa=6$ and $h=\frac{1}{256}$
    for the unsteady flow \eqref{eq:divFreeVelLeveque}
    over the time interval $[0,t_{e}]$
    where $t_{e}=0.3$, $0.5$, $0.8$, $1$. 
    The color indicates the altitude.}
  \label{fig:divFreeVelLevequeMoving}
\end{figure}

The evolution of $\mathcal{S}(t)$ at different time instants 
 and the corresponding generating cycles
 for the unsteady flow \eqref{eq:divFreeVelLeveque}
 are respectively plotted in the first and the second rows
 of Figure \ref{fig:divFreeVelLevequeMoving}. 
We also present errors and convergence rates of our LFC algorithm 
 for two different scalar functions in Table \ref{ta:taylar2},
 where all of the desired convergence rates $\kappa=2$, $4$, $6$
 are achieved. 
In particular,
 errors of the constant scalar $f=1$
 are much smaller than those of (\ref{eq:scalarLeveque})
 and 
 errors of the higher-order algorithms
 are much smaller than those of the second-order algorithm
 by orders of magnitude.

\subsection{A steady compressible flow}

In this test, the velocity field is 
\begin{equation}
  \label{eq:rotationStrainVel}
  \mathbf{u} (\mathbf{x},t)=
  \begin{bmatrix}
    x+2\pi(y+z)\\ 
    -2\pi(x+z)+y\\ 
    z+2\pi(y-x)
  \end{bmatrix}
  \quad 
\end{equation}
with a nonzero divergence $\nabla\cdot \mathbf{u}=3$
and the scalar function is 
\begin{equation}
  \label{eq:rotationStrainScalar}
  f(x,y,z,t) = (x^2+y^2+z^2)e^{-5t}.
\end{equation}
It is readily verified that the above $\mathbf{u}$ 
 and $f$ satisfy the conservation law 
 \eqref{eq:ConservationLaw}.

\begin{figure}
  \centering
  \captionof{table}{Errors and convergence rates of the LFC Algorithm 
  \ref{alg:MLFC} for calculating
  the flux of the scalar function \eqref{eq:rotationStrainScalar}
  through the fixed surface \eqref{eq:fixed_surf}
  in the flow of \eqref{eq:rotationStrainVel}
  during the time interval $[0,1]$. 
  See Figure \ref{fig:taylar3ForFixedSurf} 
  for a corresponding generating cycle.}
  \begin{tabular}{c|ccccccc} 
    \hline\hline 
  $\kappa$ & $E_{\kappa}\left(\frac{1}{32}\right)$ &  
  ${\mathcal{O}}_{\kappa}$ & $E_{\kappa}\left(\frac{1}{64}\right)$ &  
  ${\mathcal{O}}_{\kappa}$ & $E_{\kappa}\left(\frac{1}{128}\right)$ &
  ${\mathcal{O}}_{\kappa}$ & $E_{\kappa}\left(\frac{1}{256}\right)$\\
    \hline 
  2 & 4.48e-05 & -1.89 & 1.66e-04 
  &   1.41 & 6.27e-05 &   1.79 & 
  1.81e-05 \\ 
  4 & 5.38e-06 &   3.20 & 5.84e-07 
  & 3.77 & 4.29e-08 &   3.91 &
  2.86e-09\\ 
  6 & 5.74e-09 &   6.45 & 6.55e-11 
  & 6.28 & 8.41e-13 &   4.70 
  & 3.24e-14\\ 
    \hline\hline 
  \end{tabular} 
  \label{ta:taylar3ForFixedSurf}
  \vspace{3mm}
  \begin{minipage}{0.45\textwidth}
    \centering  
    \includegraphics[width=\linewidth,height=0.25\textheight]
    {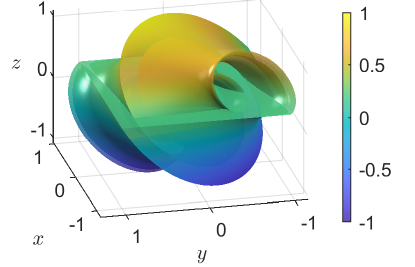}  
    \label{fig:GeneratingSurfForFixed}
  \end{minipage}
  \begin{minipage}{0.4\textwidth}  
    \centering  
    \includegraphics[width=\linewidth,height=0.25\textheight]
    {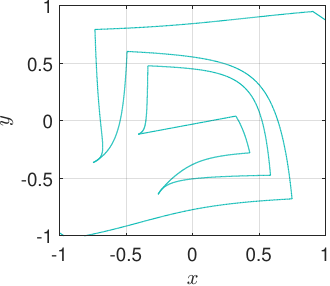}  
    \label{fig:Contour} 
  \end{minipage} 
  \captionsetup{font={small}}
  \captionof{figure}{
    The generating cycle ${\mathcal G}_{\mathcal S}$ (the left subplot)
    and part of its self-intersection in \eqref{eq:SelfInterSet}
    (the right subplot)
    constructed by Algorithm \ref{alg:generatingCycle} 
    with $\kappa=6$ and $h= \frac{1}{256}$ 
    for the fixed surface \eqref{eq:fixed_surf} 
    in the compressible flow \eqref{eq:rotationStrainVel}
    during the time interval $[0,1]$.
    The color indicates the altitude. 
}
\label{fig:taylar3ForFixedSurf}
\end{figure}

\subsubsection{A static surface}

The generating cycle $\mathcal{G}_{\mathcal{S}}$
 of the static surface 
 \begin{equation}
  \label{eq:fixed_surf}
  \mathcal{S}(u,v,t)=\left(2u-1,2v-1,\frac{1}{4}\right),\quad u,v\in(0,1)
 \end{equation}
 during the time interval $[0, 1]$
 is plotted in Figure \ref{fig:taylar3ForFixedSurf},
 where we also show
 part of the self-intersections of $\mathcal{G}_{\mathcal{S}}$, 
 \begin{equation}
   \label{eq:SelfInterSet}
   \mathcal{I}_{\mathcal{S}}:=\cup_{i=2}^{6}
   \left(\mathcal{S}\cap\mathcal{P}_{i}\right),
 \end{equation}
 where ${\mathcal S}$ and $\{\mathcal{P}_{i}\}_{i=2}^{6}$
 constitute $\mathcal{G}_{\mathcal{S}}$; 
 see \eqref{eq:partitionGeneratingCycle}.

LFC via the identity \eqref{eq:FluxIdentity} 
 requires decomposing the whole space into 
 several simple connected regions 
 and identifying the weights of each region. 
However,
 the generating cycle in Figure \ref{fig:taylar3ForFixedSurf}
 has nontrivial self-intersections.
As discussed in Section \ref{sec:introduction},
 it would be very complicated and expensive
 to decompose the bounded complement of the generating cycle
 into donating regions with nonzero indices.
Furthermore,
 the decomposition requires calculating the intersection
 of spline surfaces,
 which can be arbitrarily ill-conditioned.
In contrast,
 our LFC algorithm is based on identity \eqref{eq:FluxIdentity1}
 and thus avoids the decomposition. 
The excellent conditioning of our LFC algorithm
 is shown by the last row of Table 
 \ref{ta:taylar3ForFixedSurf},
 where the second-, fourth-, and sixth-order
 convergence rates are once again demonstrated.

\subsubsection{A moving surface}
Now, we test our LFC algorithm for a moving surface, 
\begin{equation}
  \label{eq:MS2}
  \mathcal{S}(u,v,t)=\left(u\sin\frac{\pi t}{2},v\sin\frac{\pi t}{2},
  \frac{t^2}{72}(9u^2+4v^2)\right),\quad u,v\in(0,1).
\end{equation}
In Figure \ref{fig:taylar3},
 we plot the constructed generating cycles
 of the moving surface \eqref{eq:MS2}
 for two time intervals $[0, 1]$ and $[0, 2]$. 
As the time span becomes longer, 
 the generating cycle exhibits finer features
 and more self-intersections.
Corresponding errors and convergence rates 
 are presented in Table \ref{ta:taylar3}, 
 where it is clear that the longer time interval 
 and more complex geometry 
 do not deteriorate the accuracy of our LFC algorithm.
 \begin{figure}[htbp]
  \centering
  \captionof{table}{Errors and convergence rates of our LFC 
    Algorithm \ref{alg:MLFC} for calculating
    the flux of the scalar function \eqref{eq:rotationStrainScalar}
    through the deforming surface \eqref{eq:MS1}
    in the flow of \eqref{eq:rotationStrainVel}
    for two time intervals; 
    see Figure \ref{fig:taylar3} for the corresponding generating cycles.}
  \begin{tabular}{c|cccccccc} 
    \hline\hline 
    $\kappa$ & $E_{\kappa}\left(\frac{1}{32}\right)$& $\mathcal{O}_{\kappa} $ & $E_{\kappa}\left(\frac{1}{64}\right)$
    & ${\mathcal{O}}_{\kappa}$ &$E_{\kappa}\left(\frac{1}{128}\right)$
    & ${\mathcal{O}}_{\kappa}$ &$E_{\kappa}\left(\frac{1}{256}\right)$&
    \\ 
    \hline
    \multicolumn{8}{c}{$[t_0, t_e]=\left[ 0, 1 \right]$}
    \\
    \hline \hline 
    2 & 2.15e-02& -0.58 & 3.20e-02 & 1.53 & 1.11e-02 & 1.82 & 3.13e-03\\ 
    4 & 8.04e-04 & 3.20 &8.75e-05 & 3.77 & 6.42e-06 
    & 3.91 & 4.28e-07\\ 
    6 & 7.00e-06& 6.96 &5.62e-08 & 6.83 & 4.95e-10 
    & 6.21 & 6.71e-12\\ 
    \hline \hline 
    \multicolumn{8}{c}{$[t_0, t_e]=\left[ 0, 2 \right]$}\\
    \hline\hline 
    2 & 1.27& 5.19& 3.48e-02& -0.02& 3.52e-02 & 1.55 & 1.20e-02\\
    4 & 1.41e-02& 4.33& 6.99e-04& 3.04& 
    8.48e-05 & 3.72 & 6.44e-06\\
    6 & 1.15e-04 & 6.67& 1.13e-06& 6.32& 
    1.41e-08 & 6.15 & 1.98e-10\\
    \hline\hline
  \end{tabular}
  \label{ta:taylar3}

  \vspace{3mm}

  \centering
  \captionof{table}{CPU times $T_{\kappa}(\texttt{nNodeS})$
    consumed by our $\kappa$th-order LFC algorithm
    with $h=\frac{1}{\texttt{nNodeS}}$
    for the case $[t_0,t_{e}]=[0,1]$ in Table \ref{ta:taylar3}
    on a Lenovo YOGA 710 laptop
    with an Intel Core i7 7500U processor.
  }
  \begin{tabular}{c|ccccccc}
    \hline
    \hline
    $\kappa$& $T_{\kappa}(32)$ & ratio & $T_{\kappa}(64)$ 
    &ratio & $T_{\kappa}(128)$& ratio  & $T_{\kappa}(256)$\\
    \hline
    2& 2.58s& 2.67 & 6.89s & 4.98& 34.30s & 5.69& 195.11s\\ 
    4& 2.66s& 4.78& 12.71s &5.40& 68.67s & 6.01 & 412.95s\\
    6& 5.41s& 5.56& 30.06s &5.83& 175.11s &6.33& 1108.75s\\
    \hline
    \hline
  \end{tabular}
  \vspace{3mm}

  \includegraphics[width=0.8\linewidth,height=0.22\textheight]
  {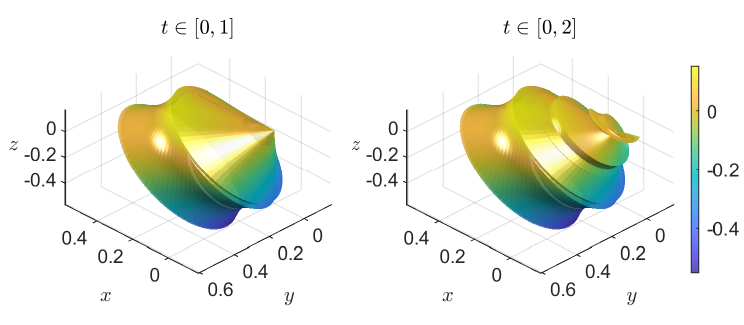}
  \captionsetup{font={small}}
  \captionof{figure}{The generating cycles
    constructed by Algorithm \ref{alg:generatingCycle}
    with $\kappa=6$ and $h= \frac{1}{256}$
    for the moving surface \eqref{eq:MS2} 
    in the compressible flow \eqref{eq:rotationStrainVel}. 
    The color indicates the altitude.} 
  \label{fig:taylar3}
\end{figure}
We also report CPU times of our LFC algorithm
 for the shorter time interval in Table \ref{ta:taylar3}. 
The linear algorithm ($\kappa=2$) consumes 195.11 seconds 
 in calculating the flux with a relative error of 3.13e-03 
 while the sixth-order algorithm only 5.41 seconds 
 in producing a result about five times more accurate. 
This drastic comparison is not a surprise 
 since the Gauss quadrature formulas and the cubic/quintic splines 
 greatly improve the cost-effectiveness of LFC. 

When we simultaneously reduce by a factor of 2
 the spatial grid size and the time step size, 
 we expect that the total CPU time of LFC increase
 by a factor of 8,  
 because the reduction of the time step size
 should increase the construction time of the generating cycle
 by a factor of 2
 and that of the spatial grid size
 should increase the integration time of the scalar
 by a factor of 4, c.f. the quadrature formula (\ref{eq:quadrature_rule}). 
The tests results in Table \ref{ta:taylar3}, however, 
 show factors around $6$ for the two finest grids. 
In comparison,
 numerically solving the scalar conservation law in three dimensions
 would increase the CPU time by a factor of 16
 every time the grid size is halved.
 
All numerical results presented in this section
 can be reproduced 
 with the companion MATLAB package freely available 
 at \texttt{https://github.com/wdachub/LFC3D}.

\section{Conclusion}
\label{sec:conclusion}

We have extended our theory of donating regions 
 for two-dimensional static curves
 \cite{zhang2013donating,zhang2015generalized,zhang2019lagrangian}
 to that for three- and higher-dimensional moving hyperspaces.
In the context of scalar conservation laws,
 the Eulerian flux through a moving hyperspace
 within a time interval
 in the flow of a nonautonomous velocity field 
 is identified with
 two Lagrangian fluxes,
 one as an integral of $f$ over the generating cycle
 and the other as a weighted sum of integrals
 over donating regions of all nonzero indices.
Depending only on the given velocity and the initial conditions, 
 both Lagrangian fluxes are time independent
 and free of solving the scalar conservation law. 
As such, the flux identities can be considered as analytic solutions
 of the problem of LFC. 
Our analysis also casts light on
 the problem of Lagrangian particle classifications. 
 
Based on the flux identity that is more suitable for numerical computation,
 we propose a simple LFC algorithm for moving surfaces 
 in three dimensions,
 prove its convergence rates, 
 and demonstrate its efficiency and accuracy
 by results of an array of numerical tests.
In particular,
 the new LFC algorithm
 is much more efficient than solving the scalar conservation laws
 and overcomes the ill-conditioning of previous LFC algorithms.  

Future studies include two directions. 
First,
 many problems in meteorology and oceanography 
 involve flux calculations in curved 2-manifolds, 
 and thus we plan to augment our theory and algorithms
 of LFC for this scenario as a further generalization
 of the current work.
Second,
 since the time interval of LFC can be arbitrarily long
 while being free
 of the Courant-Friedrichs-Lewy condition, 
 the theory and algorithms of LFC
 may lead to new efficient finite volume methods
 for numerically solving partial differential equations
 such as conservation laws and the advection-diffusion equation. 







\bibliographystyle{abbrv}

\end{document}